\documentclass[a4paper, reqno, 12pt]{amsart}

\usepackage[usenames,dvipsnames]{color}
\usepackage{amsthm,amsfonts,amssymb,amsmath,amsxtra}
\usepackage[all]{xy}
\SelectTips{cm}{}
\usepackage{xr-hyper}
\usepackage[colorlinks=
   citecolor=Black,
   linkcolor=Red,
   urlcolor=Blue]{hyperref}
\usepackage{verbatim}

\usepackage[margin=1.25in]{geometry}
\usepackage{mathrsfs}
\usepackage{kbordermatrix}
\usepackage{tikz}\pdfpageattr{/Group <</S /Transparency /I true /CS /DeviceRGB>>}
\usetikzlibrary{arrows,calc,intersections,matrix,positioning,through}
\usepackage{tkz-euclide}
\usetkzobj{all}
\usepackage{tikz-cd}
\tikzset{commutative diagrams/.cd,every label/.append style = {font = \normalsize}}
\usepackage{tkz-graph}
\usetikzlibrary{arrows,%
                shapes,positioning}

\tikzset{
  dep u/.style={insert path={-- ++(0,15) node{}}},
  dep r/.style={insert path={-- ++(15,0) node{}}},
  dep d/.style={insert path={-- ++(0,-15) node{}}},
  dep l/.style={insert path={-- ++(-15,0) node{}}},
  recurse lattice path/.code args={#1#2}{
    \ifx#1.\else\tikzset{dep #1,recurse lattice path=#2}\fi
  },
  lattice path/.style={recurse lattice path=#1.}
}

\usetikzlibrary{external}
\usepackage[all]{xy}
\usepackage[aligntableaux=center]{ytableau}
\usepackage{epstopdf}

\usepackage{array}
\usepackage{color}

\makeatletter
\tikzset{
    fading speed/.code={
        \pgfmathtruncatemacro\tikz@startshading{50-(100-#1)*0.25}
        \pgfmathtruncatemacro\tikz@endshading{50+(100-#1)*0.25}
        \pgfdeclareverticalshading[%
            tikz@axis@top,tikz@axis@middle,tikz@axis@bottom%
        ]{axis#1}{100bp}{%
            color(0bp)=(tikz@axis@bottom);
            color(\tikz@startshading)=(tikz@axis@bottom);
            color(50bp)=(tikz@axis@middle);
            color(\tikz@endshading)=(tikz@axis@top);
            color(100bp)=(tikz@axis@top)
        }
        \tikzset{shading=axis#1}
    }
}
\makeatother

\tikzexternaldisable

\usepackage{todonotes}

\RequirePackage{xspace}
\RequirePackage{etoolbox}
\RequirePackage{varwidth}
\RequirePackage{enumitem}
\RequirePackage{tensor}
\RequirePackage{mathtools}
\RequirePackage{longtable}
\RequirePackage{multirow}

\newcommand{\Mat}{\text{Mat}}
\def\a{\alpha}
\def\b{\beta}
\def\<{\langle}
\def\>{\rangle}

\newcommand{{\BG}}{\ensuremath{\mathbb {G}}\xspace}

\newcommand{{\BK}}{\ensuremath{\mathbb {K}}\xspace}

\newcommand{\BR}{\ensuremath{\mathbb {R}}\xspace}

\newcommand{\CA}{\ensuremath{\mathcal {A}}\xspace}
\newcommand{\CB}{\ensuremath{\mathcal {B}}\xspace}
\newcommand{\CC}{\ensuremath{\mathcal {C}}\xspace}

\newcommand{\CN}{\ensuremath{\mathcal {N}}\xspace}

\newcommand{\CP}{\ensuremath{\mathcal {P}}\xspace}

\newcommand{\CR}{\ensuremath{\mathcal {R}}\xspace}

\def\pos{\text{pos}}

\newcommand{\GL}{\mathrm{GL}}

\DeclareMathOperator{\Gr}{\mathrm{Gr}}

\newtheorem{theorem}{Theorem}
\newtheorem{proposition}[theorem]{Proposition}
\newtheorem{lemma}[theorem]{Lemma}
\newtheorem{lem}[theorem]{Lemma}

\newtheorem{corollary}[theorem]{Corollary}

\theoremstyle{definition}
\newtheorem{definition}[theorem]{Definition}
\newtheorem{example}[theorem]{Example}

\newtheorem{remark}[theorem]{Remark}

\usepackage{amsthm}
\newtheorem*{main}{Theorem A}
\newtheorem*{mainB}{Theorem B}

\numberwithin{equation}{section}
\numberwithin{theorem}{section}
\newcommand{\ip}{\iota_{\text{pre}}  ( \CC_{n-1, k, 2}  )}
\newcommand{\ii}{ \sigma_{k,n}^{-2} \cdot U^-(s_2s_1) \cdot \iota_{\text{inc}} ( \CC_{n-1, k-1, 2} )}

\font\pipefont=lcircle10
\def\elbow{\smash{\raise3pt\hbox{\pipefont\rlap{\rlap{\char'014}\char'016}}}}

\def\halfelbow{\smash{\raise2pt\hbox{\pipefont\rlap{\rlap{\rlap{\char'015}\phantom{\char'017}}}}}}
\def\cross{\smash{\lower5pt\hbox{\rlap{\vrule height16pt}}\raise3pt\hbox{\rlap{\hskip-8pt \vrule height0.4pt depth0pt width16pt}}}}

\ytableausetup{boxsize=16pt}
\newcommand{\Le}{\textup{\protect\scalebox{-1}[1]{L}}}
\begin{document}

\title[]{The $m=2$ Amplituhedron}

\thanks{ }
\author[Huanchen Bao]{Huanchen Bao}
\address{Department of Mathematics, National University of Singapore, Singapore.}
\email{huanchen@nus.edu.sg}

\author[Xuhua He]{Xuhua He}
\address{The Institute of Mathematical Sciences and Department of Mathematics, The Chinese University of Hong Kong, Shatin, N.T., Hong Kong.}
\email{xuhuahe@math.cuhk.edu.hk}
\keywords{Amplituhedron, total positivity, Grassmannians}
\subjclass[2010]{Primary: 14M15. Secondary: 52B99}


\begin{abstract}
The (tree) amplituhedron $\CA_{n, k, m}$  is introduced by Arkani-Hamed and Trnka in 2013 in the study of $\CN=4$ supersymmetric Yang-Mills theory. It is defined in terms of the totally nonnegative Grassmannians. In this paper, we show that the amplituhedron $\CA_{n, k, m}$ for $m=2$ admits a triangulation. Our collection of cells is  constructed via BCFW-type recursion. We also provide a diagrammatic interpretation of our construction.
\end{abstract}

\maketitle

\tableofcontents

\section*{Introduction}

\subsection{} 
The theory of totally positive real matrices was developed in the 1930’s by I.Schoenberg and independently by F.Gantmacher and M.Krein after earlier contributions by M.Fekete and G.Polya in 1912. In \cite{Lu94}, Lusztig established the theory of total positivity for any split reductive group $G$ over $\BR$ and defined totally nonnegative (partial) flag variety $(G/P)_{\ge 0}$. The cell decomposition of $(G/P)_{\ge 0}$ is parameterized by Bruhat intervals in the Weyl group $W$ of $G$. This was conjectured by Lusztig and proved by Riestch \cite{Rie1, Rie2}. Explicit parametrization of each cell is given in \cite{MR}, which plays a crucial role in this paper. 

The totally nonnegative Grassmannian $\Gr_{k,n}^{\ge 0}$ is a special case of the totally nonnegative partial flag variety $(G/P)_{\ge 0}$. 
It is the subset of $\Gr_{k,n}$ whose Pl\"{u}cker coordinates are all nonnegative.  We denote by $\Mat_{n,k}^{\ge 0}$ the set of full rank $n \times k$ matrices with nonnegative $k \times k$ minors. We may identify $\Gr_{k,n}^{\ge 0}$ with $\Mat_{n, k}^{\ge 0}/ \GL^+_{k}$, where  $\GL^+_{k}$ denotes the group of $k \times k$ real matrices with positive determinants. 

The positroid cells of $\Gr_{k, n}^{\ge 0}$ were studied extensively in Postnikov \cite{Pos}. The definitions of the positroid cells and the Rietsch cells are very different. However, it is proved (\cite{TL13}, \cite{Lam14} and \cite{Lu19}) that they agree. In \cite{Pos}, Postnikov studied several combinatorial models of the positroid cells. Among those models, the matroid description \cite[Chap. 3]{Pos} and the plabic graphs \cite[Chap. 13]{Pos} are most relevant for this paper. 

\subsection{}
Let $n,k,m \in \mathbb{Z}_{> 0}$ such that $k+m \le n$. For $Z \in \Mat^{> 0}_{k+m,n}$, we define 
\begin{align*}
	\widetilde{Z} : \Gr_{k,n}^{\ge 0} \longrightarrow \Gr_{k,k+m},\qquad 
		A \mapsto Z \cdot A. 
\end{align*}
The (tree) amplituhedron $\mathcal{A}_{n,k,m}=\mathcal{A}_{n,k,m}(Z)$ is defined to be the image of  $\Gr_{k,n}^{\ge 0}$ under the map $\tilde{Z}$. It was introduced in 2013 by Arkani-Hamed and Trnka \cite{AHT13} in the study of $\CN=4$ supersymmetric Yang-Mills theory.  

Let $\{C_\a\}$ be a collection of cells in $\Gr_{k, n}^{\ge 0}$ with $\dim C_\a =\dim \Gr_{k, k+m}$. We say that $\{C_\alpha \}$  gives a {\it triangulation} of the amplituhedron $\CA_{n, k, m}$, if for any initial data $Z \in \Mat_{k+m, n}^{>0}$, we have  

\begin{itemize}
\item Injectivity: the map $C_\a \mapsto Z \cdot C_\a$ is injective for any $\a$;

\item Disjointness: $Z \cdot C_\a \cap Z \cdot C_\b=\emptyset$ if $\a \neq \b$; 

\item Surjectivity: the union of the images $\cup_{\a} Z \cdot C_\a$ is an open dense subset of $\CA_{n, k, m}(Z)$. 
\end{itemize}

\subsection{} 
The $m=4$ amplituhedron is of immediate relevance to physics. The scattering amplitude in the planar limit for the $\CN=4$ supersymmetric Yang-Mills theory is identified as ``the Volume" of the amplituhedron at $m=4$. The scattering amplitude can be computed using the BCFW recursion introduced by Britto, Cachazo, Feng and Witten in \cite{BCFW}. The BCFW recursion represents the amplitude as a sum over basic building blocks. We refer to \cite{AHBC+16, AHT13} for more background on the physics.
	
The summand of the BCFW recursion can be translated  (\cite{AHBC+16, KWZ}), geometrically, into a collection of cells in the positive Grassmannian $\Gr_{k,n}^{\ge 0}$. We call this collection of cells the BCFW cells. It is conjectured \cite{AHT13} that the BCFW cells triangulate the amplituhedron $\mathcal{A}_{n,k,m}$ at $m=4$.

\subsection{} Although the $m=4$ case is most relevant for physics, the amplituhedron $\CA_{n,k,m}$ for arbitrary $k, n, m$ is an interesting mathematical object. In the case $k+m =n$, we have indeed $\CA_{n,k,m} \cong \Gr_{k,n}^{\ge 0}$.

The amplituhedron at $k=1$, for arbitrary $m$ and $n$, is a classical object called cyclic polytope. The triangulation of cyclic polytopes was studied by Rambau \cite{Ra97}.  All triangulations of $\CA_{n,1,m}$ have been obtained in \cite{Ra97}.
	
Karp and Williams \cite{KW19} obtained the triangulation of the amplituhedron at $m=1$. Galashin and Lam \cite{GL18} established the duality between the triangulation of $\CA_{n,k,m}$ and the triangulation of $\CA_{n, n-m-k,m}$ for even $m$, called the parity duality.
	
Karp, Williams and Zhang  \cite{KWZ} provided a nonrecursive collection of BCFW cells (at $m=4$). They proved that their images in $\CA_{n,2,4}$ (for $k=2$ only) are pairwise disjoint.  They also made a numerical conjecture \cite[Conjecture~8.1]{KWZ} on the triangulation of $\CA_{n,k,m}$ for all even $m$. 
	
There is evidence coming from both mathematics and physics indicating that the amplituhedron $\CA_{n,k,m}$ is better behaved at even $m$ than at odd $m \ge 3$. The $m=1$ case is, in some sense, a singular case from this point of view.
	
The $m=2$ amplituhedron was studied in \cite{KWZ} and \cite{AHTT}. In both papers, the images of certain specific collection of cells in $\Gr_{k,n}^{\ge 0}$ under the amplituhedron map are proved to be pairwise disjoint. 

In a very recent preprint \cite{Luk}, Lukowski studied the structure of the boundary of the $m=2$ amplituhedron and showed that the Euler characteristic of the amplituhedron equals one. In another very recent preprint \cite{LPSV}, Lukowski, Parisi, Spradlin and Volovich characterizes the generalised triangles for the $m=2$ amplituhedron. 
	
\subsection{} In this paper, we prove that the $m=2$ amplituhedron $\CA_{n,k,2}$ admits a triangulation for any initial $Z \in \Mat^{> 0}_{k+2, n}$. Our collection of cells for the triangulation is given by a BCFW-type recursion. 

\begin{main}[Theorem~\ref{thm:main}]
Let $\CC_{n-1, k, 2}$ (resp. $\CC_{n-1, k-1, 2}$) be a collection of cells in $\Gr_{k,n-1}^{\ge 0}$ (resp. $\Gr_{k-1,n-1}^{\ge 0}$) that triangulates $\CA_{n-1,k,2}$ (resp. $\CA_{n-1,k-1,2}$). Then 
\[
\CC_{n, k, 2} = \iota_{\text{pre}}( \CC_{n-1, k, 2}) \cup \sigma_{k,n}^{-2} \cdot U^-(s_2 s_1)\cdot \iota_{\text{inc}} (\CC_{n-1, k-1, 2})
\text{ triangulates } \CA_{n, k, 2}.
\]
 Moreover, the union above is a disjoint union. 
\end{main}

Diagrammatically, the main theorem reads as follows:
 $$
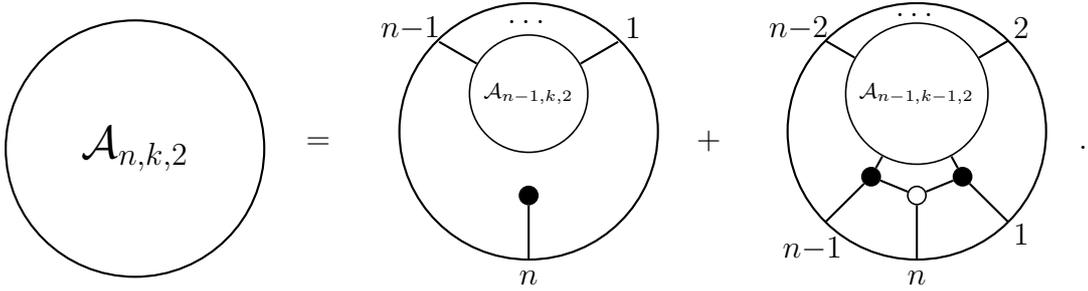
\quad\begin{tikzpicture}[baseline=(current bounding box.center)]
\pgfmathsetmacro{\radius}{1.7};
\draw[thick](0,0)circle[radius=\radius];
\node[]at (0,0) {\Large $\CA_{\scriptstyle n,k,2}$};
\end{tikzpicture}
\quad
=
\quad
\begin{tikzpicture}[baseline=(current bounding box.center)]
\pgfmathsetmacro{\radius}{1.7};
\tikzstyle{out1}=[inner sep=0,minimum size=2.4mm,circle,draw=black,fill=black,semithick]
\tikzstyle{in1}=[inner sep=0,minimum size=2.4mm,circle,draw=black,fill=white,semithick]
\draw[thick](0,0)circle[radius=\radius];

\node[circle,draw=black,fill=white,semithick] (Ak) at (0,0.5) {$\scriptstyle \CA_{ n-1,k,2}$};

\node[inner sep=0](bn-1)at(135:\radius){};
\node[xshift=-1ex] at(135:\radius+.24){\Large$\scriptstyle n-1$};

\node[inner sep=0](b1)at(135-1*90:\radius){};
\node at(135-1*90:\radius+.24){$1$};

\node[out1](in)at(-90: \radius/2){};
\node[inner sep=0](bn)at(-90:\radius){};
\node at(-90:\radius+.24){$n$};
\node[thick] at(90:\radius/2+.6){$\dots$};
\path[thick](bn.center)edge(in) (bn-1)edge(Ak) (b1)edge(Ak);
\end{tikzpicture}
\quad
+ 
\quad\begin{tikzpicture}[baseline=(current bounding box.center)]

\tikzstyle{out1}=[inner sep=0,minimum size=2.4mm,circle,draw=black,fill=black,semithick]
\tikzstyle{in1}=[inner sep=0,minimum size=2.4mm,circle,draw=black,fill=white,semithick]

\pgfmathsetmacro{\radius}{1.7};
\draw[thick](0,0)circle[radius=\radius];

\node[circle,draw=black,fill=white,semithick] (Ak-1) at (0,0.5) {$\scriptstyle \CA_{ n-1,k-1,2}$};

\node[inner sep=0](b2)at(135-1*90:\radius){};
\node at(135-1*90:\radius+.24){$2$};

\node[inner sep=0](b1)at(135-2*90:\radius){};
\node at(135-2*90:\radius+.24){$1$};

\node[inner sep=0](bn)at(-90:\radius){};
\node at(-90:\radius+.24){$n$};

\node[inner sep=0](bn-1)at(-135:\radius){};
\node[yshift=-1ex] at(-135:\radius+.24){\Large$\scriptstyle n-1$};

\node[inner sep=0](bn-2)at(135:\radius){};
\node[xshift=-1ex] at(135:\radius+.24){\Large$\scriptstyle n-2$};
\node[out1](i1)at(-45:\radius/2){};
\node[out1](in-1)at(-135:\radius/2){};
\node[in1](in)at(-90: \radius/2){};

\path[thick](b1.center)edge(i1) (b2.center)edge(Ak-1) (in-1)edge(bn-1.center) (in)edge(bn.center) (Ak-1)edge(bn-2.center) (i1)edge(in) (in)edge(in-1) (i1)edge  (Ak-1) (in-1)edge (Ak-1);

\node[thick] at(90:\radius/2+.7){$\dots$};
\end{tikzpicture}\quad .
$$
We refer to \S\ref{5-diag} for detailed discussion on the diagrammatic expression. 

\smallskip

We prove the triangulation by induction on both $k$ and $n$, where the base cases are $k=1$, i.e., the cyclic polytope cases. Our strategy relies on Marsh-Rietsch's parametrization \cite{MR} of cells in $\Gr_{k,n}^{\ge 0}$, as well as the cyclic symmetry $\sigma_{k,n}$ of $\Gr_{k,n}^{\ge 0}$, to perform reductions on both $k$ and $n$. Our proof of the disjointness is inspired by \cite{AHTT}. 

\subsection{}

We then use the algorithm in Theorem~A to find an explicit (non-recursive) collection of cells that triangulates the amplituhedron $\CA_{n,k,2}$. This collection matches  the collection of cells considered in \cite{KWZ, AHTT}.

\begin{mainB}[Theorem~\ref{thm:m=2}]
The cells
$$
\{\CP^J_{k, n}(a_1, a_2, \ldots, a_k)  \subset \Gr_{k,n}^{\ge 0} \vert  1<a_1<a_2<\ldots<a_k \le n-1\}
$$
 gives a triangulation of the amplituhedron $\mathcal{A}_{n,k,2} = \mathcal{A}_{n,k,2}(Z)$ for any initial data $Z \in \text{Mat}_{k+2, n}^{>0}$.

\end{mainB}
\subsection{}The paper is organized as follows. We recall totally nonnegative Grassmannians and their cellular decomposition in \S\ref{1-pre}.  We state the main result in \S\ref{2-main} with examples. We prove the main theorem in \S\ref{sec:pre} and \S\ref{sec:proof}. We then compute an explicit collection of cells for the triangulation of $\CA_{n,k,2}$ in \S\ref{4-explicit} using the main theorem. We translate our result into diagrams in \S\ref{5-diag}.

{\bf Acknowledgment:} We thank Lauren Williams for helpful discussions. We thank George Lusztig and Thomas Lam for useful comments on the paper. HB is supported by a NUS start-up grant. XH is partially supported by NSF grant DMS-1801352 and a start-up grant of CUHK. The plabic graphs in the paper are modified over the graphs from \cite{KWZ}.

\section{Totally nonnegative Grassmannians}\label{1-pre}

Throughout the paper, we identify the real algebraic variety with its $\BR$-points. In other words, we simply write $\GL_n$ for $\GL_n(\BR)$, etc. In this section, we recall the definition and properties of totally nonnegative Grassmannians. We start  with Postnikov's approach, and then Lusztig's theory of totally positivity, and then Marsh-Rietsch's parametrization of cells. 

\subsection{The Grassmannian and its totally nonnegative parts} \label{subsec:ele}
Let $k<n$ be positive integers. The (real) Grassmannian $\Gr_{k, n}$ is the space of all $k$-dimensional subspaces in $\BR^n$. 
Let $\Mat_{n, k}^\times$ be the set of (real) $n \times k$ matrices with rank $k$. We represent the elements in the Grassmannian in terms of matrices as follows. Note that we have a natural action of $\GL_k$ on $\Mat^\times_{n, k}$ by right multiplication and a natural map $\Mat_{n, k}^\times \to \Gr_{k, n}$ which sends a $n \times k$ matrix $M$ with rank $k$ to the $k$-dimensional subspace of $\BR^n$ spanned by the columns of $M$. It is easy to see that this map induces an isomorphism of real algebraic varieties 
$$
\Mat_{n, k}^\times / \GL_k \cong \Gr_{k, n}.
$$

Following Postnikov \cite{Pos}, we say that $M \in \Mat_{n, k}^\times$ is {\it totally positive} (resp. {\it totally nonnegative}) if all the maximal minors of $M$ are positive (resp. nonnegative). We denote by $\Mat_{n, k}^{>0}$ (resp. $\Mat_{n, k}^{\ge 0}$) the set of totally positive (resp. totally nonnegative) elements in $\Mat_{n, k}^\times$. 

Let $\GL_{k}^+$ be the group of $k\times k$-matrices with positive determinant. The totally positive Grassmannian $\Gr_{k, n}^{>0}$ and the totally nonnegative Grassmannian $\Gr_{k, n}^{\ge 0}$ are defined respectively as
\[
	\Gr_{k, n}^{>0} = \Mat_{n, k}^{> 0} / \GL^{+}_{k} \subset \Mat_{n, k}^{\times} / \GL_{k}, \quad  \Gr_{k, n}^{\ge 0} = \Mat_{n, k}^{\ge 0} / \GL^{+}_{k} \subset \Mat_{n, k}^{\times} / \GL_{k}.
\]

Let $[n] = \{ 1, 2, \dots, n\}$, and denote by $\begin{pmatrix} [n]\\ k \end{pmatrix}$ the set of $k$-element subsets of $[n]$. For any $I = \{i_1, i_2, \dots, i_k\} \in \begin{pmatrix} [n]\\ k \end{pmatrix}$, we denote by $\Delta_{I}(M)$ the maximal minor of $M \in \Mat^{\times}_{n, k}$ with respect to the $k$-rows $1\le i_1 < i_2< \cdots < i_{k} \le n$.
We shall often use notations $\Delta_{1, \ldots}(M) = \Delta_{1,i_2,\ldots, i_{k}}(M)$ and $\Delta_{\ldots, n}(M) =  \Delta_{i_1,i_2,\ldots, n}(M)$, etc..

A matroid $\mathcal{M}_x$ of $x \in \Gr_{k,n}$ is defined as 
\[
	\mathcal{M}_x = \Big\{ I \in \begin{pmatrix}	[n] \\ k\end{pmatrix} \Big\vert \Delta_{I}(x) \neq 0\Big\} \subset  \begin{pmatrix}	[n] \\ k\end{pmatrix}. 
\]	
We call $\mathcal{M}_x$ a positroid, if $x \in \Gr_{k,n}^{\ge 0}$. 

Let $\mathcal{M}$ be a positroid (of some $x \in \Gr_{k,n}^{\ge 0}$), we denote the positroid cell associated with $\mathcal{M}$ by 
\begin{equation}\label{eq:positroid}
	\mathcal{S}^{> 0}_{\mathcal{M}} = \{ x \in  \Gr_{k,n}^{\ge 0} \vert \mathcal{M}_x = \mathcal{M} \}.
\end{equation}
It is known (\cite[Theorem~3.5]{Pos}) that $\mathcal{S}^{> 0}_{\mathcal{M}} \cong \mathbb{R}^d_{\ge 0}$ for some $d$. We then have the stratification of $\Gr_{k,n}^{\ge 0}$ by positroid cells 
\[
	\Gr_{k,n}^{\ge 0} = \bigsqcup_{\mathcal{M}} \mathcal{S}^{> 0}_{\mathcal{M}}, \quad \text{ disjoint union over all positroids }\mathcal{M}.
\] 
Moreover, the closure of a positroid cell is a union of positroid cells. 

\subsection{The flag variety and its connection with the Grassmannian}\label{subsec:Lu}
Let $G=\GL_n$ and $I=\{1, 2, \ldots, n-1\}$ be the index set of the simple roots of $G$. 

Let $B^+ \subset G$ be the subgroup consisting of upper triangular matrices and $B^- \subset G$ be the subgroup consisting of lower triangular matrices. Both $B^+$ and $B^-$ are Borel subgroups of $G$ and they are conjugate in $G$. Let $\CB \cong G/B^+$ be the flag variety of $G$, i.e., the variety of Borel subgroups of $G$. We may identify $\CB$ as 
$$
\CB=\{0=V_0 \subset V_1 \subset \cdots \subset V_n=\BR^n \vert \dim V_i=i\}.
$$
 For any $g \in G$ and $B \in \CB$, we write $g \cdot B$ instead of $g B g^{-1}$. 

Let $W=S_n$ be the Weyl group of $G$. Let $\ell$ be the length function on $W$ and $\le$ be the Bruhat order on $W$. For $i \in I$, set $s_i=(i, i+1)$ be the $2$-cycle associated to $i$. The set $\{s_i\}_{i \in I}$ is the set of simple reflections of $W$. 

It is known that the $G$-orbits on $\CB \times \CB$ (for the diagonal action) are in natural bijection with $W$. For $B, B'$ in $\CB$, there is a unique $w \in W$, such that $(B, B')$ is in the $G$-orbit of $(B^+, \dot w \cdot B^+)$, where $\dot w$ is a representative of $w$ in $G$ (such orbit is independent of the choice of the representative $\dot w$). We write $\pos(B, B')=w$. For any $w, w' \in W$, define
$$\CR_{w, w'} =\{B \in \CB \mid \pos(B^+, B)=w', \pos(B^-, B)=w_0
w \}.$$

It is known that $\CR_{w, w'}$ is nonempty if and only if $w \le w'$ in the Bruhat order in $W$. We have $$ \CB=\bigsqcup_{w \le w'} \CR_{w, w'}.$$

Let $k<n$ be a positive integer and $J=I-\{k\}$. Let $P_J \subset G$ be the subgroup of $G$ consisting of block upper triangular matrices of block size $k$ and $n-k$. This is the standard parabolic subgroup of $G$ corresponding to the subset $J$ of $I$. Set $\CP_J=G/P_J$. We identify $\CP_J$ with $\Gr_{k, n}$ as $G$-homogeneous spaces by identifying the base point
\begin{align*}
		P_J \mapsto  \begin{bmatrix}
						1 &  \cdots & 0 \\
						\vdots &   \ddots &\vdots\\ 
						0 &  \cdots & 1\\
						\vdots   & \ddots &\vdots\\ 
						0 &  \cdots & 0\\
					\end{bmatrix} \in \Mat^\times_{k,n}.
\end{align*}

The embedding $B^+ \subset P$ induces a natural projection map $\pi_J: \CB \to \Gr_{k, n}$, which sends the flag $0=V_0 \subset V_1 \subset \cdots \subset V_n=\BR^n$ to the $k$-dimensional subspace $V_i$ of $\BR^n$. 

Let $W_J=S_k \times S_{n-k}$ be the Weyl group of $P$. Let $W^J \subset W$ be the set of minimal length representatives in their cosets $W/W_J$. For any $w \in W$ and $w' \in W^J$ with $w \le w'$, we have the projective Richardson variety $\CP^J_{w, w'}=\pi_J(\CR_{w, w'})$. We have $$\CP^J=\sqcup_{w \in W, w' \in W^J; w \le w'} \CP^J_{w, w'}.$$

\subsection{Lusztig's theory of totally positivity} Now we recall Lusztig's theory of totally positivity on reductive groups and their flag varieties. We will only deal with the group $\GL_n$ instead of any reductive group here. The readers interested in the general theory may refer to Lusztig's papers \cite{Lu94}. 

For $1 \le i \le n-1$ and $a \in \mathbb{R}$, we set 
\[x_i(a) = \kbordermatrix{
   	&    		&		& i          	& i+1 	     & 			&	\\
   	& 1 		& \cdots 	& 0		& 0	   & \cdots 	& 0	\\
   	& \vdots	&  \ddots 	& \vdots	& \vdots		     &		\ddots	&\vdots\\
 i 	&0 		& \cdots 	& 1 		& a 		     & \cdots 	& 0\\
 i+1  	&0		& \cdots 	&0 		& 1		     &	\cdots	&0 \\	
 	&\vdots 	&\ddots	&\vdots	&\vdots	     & \ddots 	&\vdots\\
	&0	&\cdots	&0	&0	     & \cdots 	&0
},\quad
y_i(a) = \kbordermatrix{
   	&    		&		& i          	& i+1 	     & 			&	\\
   	& 1 		& \cdots 	& 0		& 0	   & \cdots 	& 0	\\
   	& \vdots	&  \ddots 	& \vdots	& \vdots		     &		\ddots	&\vdots\\
 i 	&0 		& \cdots 	& 1 		& 0 		     & \cdots 	& 0\\
 i+1  	&0		& \cdots 	&a 		& 1		     &	\cdots	&0 \\	
 	&\vdots 	&\ddots	&\vdots	&\vdots	     & \ddots 	&\vdots\\
	&0	&\cdots	&0	&0	     & \cdots 	&0
}.
\]

Let $U^+$ be the unipotent radical of $B^+$, i.e., the subgroup of $B^+$ consisting of all upper triangular matrices with diagonal entries $1$. Let $U^-$ be the unipotent radical of $B^-$, i.e., the subgroup of $B^-$ consisting of all lower triangular matrices with diagonal entries $1$.

For any $w \in W$ with a reduced expression $w = s_{i_1} s_{i_2} \cdots s_{i_N}$, we define 
\begin{gather*}
U^+(w)=\{x_{i_1}(a_1) \cdots x_{i_N}(a_N) \vert a_j>0 \text{ for all } 1 \le j \le N\}; \\
U^-(w)=\{y_{i_1}(a_1) \cdots y_{i_N}(a_N) \vert a_j>0 \text{ for all } 1 \le j \le N\}.
\end{gather*}

It is known \cite[Proposition~2.7]{Lu94} that the sets $U^+(w)$ and $U^-(w)$ are independent of the reduced expression of $w$.
Let $w_0$ be the longest element in the Weyl group $W$. We write $U^+_{>0}= U^+(w_0)$ and $U^-_{>0}= U^-(w_0)$.

It is proved in \cite[Theorem~8.7]{Lu94} that the subsets $U^+_{>0} \cdot B^-$ and $U^-_{>0} \cdot B^+$ of $\CB$ are the same. This is the totally positive part $\CB_{>0}$ of the flag variety $\CB$. The totally nonnegative part $\CB_{\ge 0}$ of $\CB$ is defined to be the closure of $\CB_{>0}$ in $\CB$. For any $w, w' \in W$ with $w \le w'$, we set $\CB_{w, w'; >0}=\CB_{w, w'} \cap \CB_{\ge 0}$. Then 
$$
\CB_{\ge 0}=\sqcup_{w, w' \in W; w \le w'} \CB_{w, w'; >0}.
$$

Let $k<n$ be a positive integer and $J=I-\{k\}$. We have $\Gr_{k, n}^{\ge 0}=\pi_J(\CB_{\ge 0})$ and $\Gr_{k, n}^{>0}=\pi_J(\CB_{>0})$. This definition of the totally nonnegative Grassmannian (resp. totally positive Grassmannian) agrees with the definition in \S\ref{subsec:ele} (cf. \cite[Theorem~3.6]{Lam14}, \cite[Theorem~7.8]{TL13}, \cite[Theorem~0.5]{Lu19}).

For any $w \in W$ and $w' \in W^J$ with $w \le w'$, we set $\CP^J_{w, w'; >0}=\pi_J(\CB_{w, w'; >0})$. By \cite[Section~7]{Rie2}, 
$$\Gr_{k, n}^{\ge 0}=\sqcup_{w \in W, w' \in W^J; w \le w'} \CP^J_{w, w'; >0}.
$$ 
Moreover, for any $w \in W$ and $w' \in W^J$ with $w \le w'$, the map $\pi_J$ induces an isomorphism of real semi-algebraic varieties $\CB_{w, w'; >0} \cong \CP^J_{w, w'; >0}$.

\subsection{The Marsh-Rietsch parametrization} For $i \in I$, set 
$$
\dot{s_i}=x_i(-1)y_i(1)x_i(-1)= \kbordermatrix{
   	&    		&		& i          	& i+1 	     & 			&	\\
   	& 1 		& \cdots 	& 0		& 0	   	     & \cdots 	& 0	\\
   	& \vdots	&  \ddots 	& \vdots	& \vdots	     &\ddots	&\vdots\\
 i 	&0 		& \cdots 	& 0 		& -1 		     & \cdots 	& 0\\
 i+1  	&0		& \cdots 	&1 		& 0		     &	\cdots	&0 \\	
 	&\vdots 	&\ddots	&\vdots	&\vdots	     & \ddots 	&\vdots\\
	&0		&\cdots	&0		&0	     	     & \cdots 	&1
}.
$$ 

Let $w \le w'$. We fix a {\it reduced expression} $\underline w'$ of $w'$ with factors $(s_{i_1}, s_{i_2}, \ldots, s_{i_l})$, i.e. $w'=s_{i_1} s_{i_2} \cdots s_{i_l}$ and $l=\ell(w)$. A {\it subexpression} for $w$ in $\underline w'$ is an expression $\underline w=(w_{\{0\}}, \ldots, w_{\{l\}})$, where $w_{\{0\}}=1$, $w_{\{l\}}=w$ and for any $j$, $w_{\{j\}} \in \{w_{\{j-1\}}, w_{\{j-1\}} s_{i_j}\}$. The subexpression $\underline w$ is called {\it positive} if $w_{\{j-1\}}<w_{\{j-1\}} s_{i_j}$ for all $j$. It is known \cite[Lemma 3.5]{MR} that the positive subexpression for $w$ in $\underline w'$ exists and is unique. We denote by $\underline w_+$ the unique positive subexpression for $w$ in $\underline w'$. Set 
$$
G^{>0}_{\underline w_+, \underline w'}=\{g_1 \cdots g_l \vert g_j \in y_{i_j}(\BR_{>0}) \text{ if } w_{j_1}=w_j, g_j=\dot s_{i_j} \text{ if } w_{j-1}<w_j\}.
$$
It is proved in \cite[Theorem 11.3]{MR} that 

\begin{theorem}
The map $g \mapsto g \cdot B^+$ induces an isomorphism of real semi-algebraic varieties $$G^{>0}_{\underline w_+, \underline w'} \cong \CR_{w, w'; >0}.$$
\end{theorem}

If moreover $w' \in W^J$, then the map $g \mapsto g \cdot P_J$ induces an isomorphism of the real semi-algebraic groups $$G^{>0}_{\underline w_+, \underline w'} \cong \CP^J_{w, w'; >0}.$$
In particular, $\CP^J_{w, w';>0}$ is a semi-algebraic cell of dimension $\ell(w')-\ell(w)$. 

\smallskip

Finally, we give an explicit description of the elements in $W^J$. For $0 \le a \le n-1$ and $1 \le b \le n-1$, we write $$s_{[a, b]}=\begin{cases} s_a s_{a-1} \cdots s_b, & \text{ if } a \ge b; \\ 1, & \text{ otherwise}. \end{cases}$$ The elements in $W^J$ are of the form $(s_{[a_1, 1]}) (s_{[a_2, 2]}) \cdots (s_{[a_k, k]})$, where $0 \le a_1<a_2<\ldots<a_k$. Note that the expression given above are reduced expressions, which we fix in this paper. 

\subsection{Translation between different parametrizations}\label{sub:ex}
Note that Postnikov's cells and Rietsch's cells are defined in very different ways. It is a rather nontrivial fact that these cells coincide. This is showed by Rietsch (unpublished) in 2009, by Talaska and Williams \cite{TL13}, by Lam \cite{Lam14}, and most recently by Lusztig \cite{Lu19} via different methods. In this subsection, we explain by example the translation between Rietsch's cell decomposition of $\Gr_{k,n}^{\ge 0}$ and the Postnikov's positroid cells decomposition. More details can be found in \cite{Pos}.

\begin{example}
Let us consider the cell $\CP^J_{s_1, s_1s_3s_2;>0} \subset \Gr_{2,4}^{\ge 0}$. In terms of the matrices, we have 
\begin{align*}
	\CP^J_{s_1, s_1s_3s_2;>0} &= \left\{ \dot{s}_1 y_3(a_3) y_2(a_2) \cdot 
		\begin{bmatrix}
			1& 0 \\
			0 & 1\\
			0& 0\\
			0& 0
		\end{bmatrix} \middle \vert  a_2>0, a_3 >0.
		 \right \} \\
		 &=  \left\{
		\begin{bmatrix}
			0& -1 \\
			1 & 0\\
			0& a_2\\
			0& a_2a_3
		\end{bmatrix} \middle \vert  a_2>0, a_3 >0. \right \}
\end{align*}
Hence the positroid description of the cell can be obtained as ($M \in \Mat^{\ge 0}_{4,2}$)
\[
	\CP^J_{s_1, s_1s_3s_2;>0} = \left \{ \GL^{>0}_{n} \cdot M \in \Gr_{2,4}^{\ge 0} \middle \vert 
		\begin{array}{l}
			\Delta_{12}(M)>0, \Delta_{13}(M) =0, \Delta_{14}(M)=0, \\
			\Delta_{23}(M)>0, \Delta_{24}(M)  >0, \Delta_{32}(M)=0.
			\end{array}
				\right\}
\]
\end{example}

\subsection{Embeddings}\label{sub:emb}
We define $k$-preserving embedding of totally nonnegative Grassmannians
\begin{gather*}
\iota_{\text{pre}} : \Gr_{k, n-1}^{\ge 0} \longrightarrow \Gr_{k, n}^{\ge 0}, \qquad 
A \in \text{Mat}^{\ge 0}_{n-1 \times k} \mapsto \left[\begin{array}{c}
A \\
\hline
0 \cdots 0 
\end{array}\right] \in \text{Mat}^{\ge 0}_{n, k},
\end{gather*}

and $k$-increasing embedding of  totally nonnegative Grassmannians

\begin{gather*}
\iota_{\text{inc}}  :\Gr_{k-1, n-1}^{\ge 0} \longrightarrow \Gr_{k, n}^{\ge 0},\qquad 
A \in \text{Mat}^{\ge 0}_{n-1 \times k-1} \mapsto 
\left[\begin{array}{c|c}
\begin{matrix}1\end{matrix}   &\begin{matrix} 0 &\cdots &0\end{matrix} \\
\hline
\begin{matrix} 0 \\ \vdots \\ 0\end{matrix}  & {\Large{A}}
\end{array}\right] \in \text{Mat}^{\ge 0}_{n, k}.
\end{gather*}
It is clear that the image of a cell is still a cell under both maps.
In particular, we have 
\begin{align*}
\CP^J_{k,n-1} &:= \CP^J_{1, s_{[n-k-1, 1]} s_{[n-k, 2]} \cdots s_{[n-2, k]}; >0} = \iota_{\text{pre}} (\Gr_{k,n-1}^{> 0}),\\
 \CP^J_{k-1,n-1} &:=\CP^J_{1,  s_{[n-k+1, 2]} \cdots s_{[n-1, k]}; >0} = \iota_{\text{inc}}(\Gr_{k-1,n-1}^{> 0} ).
\end{align*}
We remark that these are the $k$-preserving and $k$-increasing operations considered by physicists \cite{AHBC+16}. 

\subsection{Cyclic symmetry}

Let 
\[
\sigma_{k,n} = \begin{bmatrix}
	0 &0 &\cdots &0 &(-1)^{k-1}\\
	1 &0 &\cdots &0 &0\\
	0 &1 &\cdots &0 &0\\
	\vdots& \vdots & \ddots & \vdots& \vdots\\
	0 &0 &\cdots &1 &0
\end{bmatrix} \in \text{Mat}_{n \times n}.
\]

For any $M \in \Mat_{n, k}$, the matrix $\sigma_{k,n}  \cdot M$ is obtained by cyclically shifting the rows of $M$ and multiplying the first row by $(-1)^{k-1}$. So we have 
\begin{equation}\label{eq:cyclic}
\sigma_{k,n} \cdot M \in\Mat^{\ge 0}_{n, k}, \quad \text{ if } M \in\Mat^{\ge 0}_{n, k}.
\end{equation}

It follows immediately from \eqref{eq:positroid} that $\sigma_{k,n} \cdot \mathcal{S}^{> 0}_{\mathcal{M}}$ is also a positroid cell of $\Gr_{k,n}^{\ge 0}$ if $\mathcal{S}^{> 0}_{\mathcal{M}}$ is. The following proposition will be used in Proposition~\ref{prop:Ra97}.

\begin{proposition}\label{prop:cyc}
Let $n-k\ge 2$. We have $$\sigma^{-2}_{k,n} \cdot U^-(s_2s_1) \cdot \CP^J_{k-1,n-1} = \CP^J_{s_{[n-2, k+1]}, s_{[n-k,1]}\cdots s_{[n-2, k-1]}s_{[n-1, k]}; >0}.$$
\end{proposition}

\begin{proof}
Since $n -k \ge 2$, we have
\[
	U^-(s_2s_1) \cdot  \CP^J_{k-1,n-1} =   \CP^J_{1, s_{[2,1]}s_{[n-k+1,2]}\cdots s_{[n-1, k]};>0},
\]
and 
\begin{align*}
	 &\CP^J_{1, s_{[2,1]}s_{[n-k+1,2]}\cdots s_{[n-1, k]};>0} \\
	 =& \left\{ M \in \Mat^{\ge 0}_{{n, k}} / \GL_{k}^{+} \middle\vert \begin{array}{l} \Delta_{(1, \dots)}(M)>0, \Delta_{(2, \dots)}(M)>0,\\ \Delta_{(3, \dots)}(M) >0,  \Delta_{\text{otherwise}}(M)=0.\end{array} \right\}.
\end{align*}

Then we have 
\begin{align*}
	 &\sigma_{k,n}^{-2} \cdot \CP^J_{1, s_{[2,1]}s_{[n-k+1,2]}\cdots s_{[n-1, k]};>0} \\
	 =& \left\{ M \in \Mat^{\ge 0}_{n, k} / \GL_{k}^{+} \middle \vert \begin{array}{l} \Delta_{(\dots,n-1)}(M)>0, \Delta_{( \dots,n)}(M)>0,\\ \Delta_{(1, \dots)}(M) >0,  \Delta_{\text{otherwise}}(M)=0.\end{array} \right\}\\
	 = & \CP^J_{s_{[n-2, k+1]}, s_{[n-k,1]}\cdots s_{[n-2, k-1]}s_{[n-1, k]}; >0}.
\end{align*}
This finishes the proof.
\end{proof}

\section{Main result}\label{2-main}

\subsection{The amplituhedron} 

\begin{definition}\label{def:amp}
Let $Z \in \Mat^{> 0}_{k+m,n}$, with $m \ge 0$ and $k +m \le n$. Then we define 
\[
	\widetilde{Z} : \Gr_{k,n}^{\ge 0} \longrightarrow \Gr_{k,k+m},\quad 
		A \mapsto Z \cdot A.
\]
The {\it (tree) amplituhedron} $\mathcal{A}_{n,k,m}(Z)$ is defined to the image of  $\Gr_{k,n}^{\ge 0}$ under the map $\widetilde{Z}$. 
\end{definition}

Note that $\widetilde{Z}$ is a well-defined, $\GL_{k+m}$-equivariant and proper map.

\begin{definition}\label{def:tri}
Let $\{C_\a\}$ be a collection of cells in $\Gr_{k, n}^{\ge 0}$ with $C_\a=\dim \Gr_{k, k+m}$. We say that $\{C_\alpha \}$  gives a {\it triangulation} of the amplituhedron $\CA_{n, k, m}$ if for any initial data $Z \in \Mat_{k+m, n}^{>0}$, we have  

\begin{itemize}
\item Injectivity: the map $C_\a \mapsto Z \cdot C_\a$ is injective for any $\a$;

\item Disjointness: $Z \cdot C_\a \cap Z \cdot C_\b=\emptyset$ if $\a \neq \b$; 

\item Surjectivity: the union of the image $\cup_{\a} Z \cdot C_\a$ is an open dense subset of $\CA_{n, k, m}(Z)$. 
\end{itemize}
\end{definition}

\begin{remark}
One can also consider the weaker notation of the $Z$-triangulation, for a fixed initial data $Z \in \Mat_{k+m,n}^{> 0}$. However, it is expected the triangulation should be independent of the initial data (at least for even $m$).
\end{remark}

\subsection{The main theorem} 

\begin{theorem}\label{thm:main}
Let $\CC_{n-1, k, 2}$ (resp. $\CC_{n-1, k-1, 2}$) be a collection of cells in $\Gr_{k,n-1}^{\ge 0}$ (resp. $\Gr_{k-1,n-1}^{\ge 0}$) that triangulates $\CA_{n-1,k,2}$ (resp. $\CA_{n-1,k-1,2}$). Then 
\[
\CC_{n, k, 2} = \iota_{\text{pre}}( \CC_{n-1, k, 2} ) \cup \sigma_{k,n}^{-2} \cdot U^-(s_2 s_1)\cdot \iota_{\text{inc}} ( \CC_{n-1, k-1, 2})
\text{ triangulates } \CA_{n, k, 2}.
\]
Moreover, the union above is a disjoint union. 
\end{theorem}

\begin{remark}
$\Gr_{0,n}^{\ge 0}$ is a point, and $\iota_{\text{inc}} ( \CC_{n-1, 0, 2}) = \CP^J_{1,1;>0}$ by convention.
\end{remark}

Our collection of cells is constructed recursively in a similar way as the BCFW recursion for $m=4$.  We also provide a diagrammatic reformulation in \S\ref{5-diag} of our main theorem following the diagrammatic BCFW-type recursion. We remark that our recursion is motivated by the reformulation of the BCFW recursion in \cite{BH15}. The proof of the main theorem shall be given in \S\ref{sec:proof}. 

\subsection{The numerical result}
The cardinality of the set $\CC_{n,k,2} $ was conjectured in \cite[Conjecture~8.1]{KWZ}, which now follows from the main theorem. 

\begin{proposition}
We have $\vert  \CC_{n,k,2} \vert = \begin{pmatrix}n-2 \\ k\end{pmatrix}$. 
\end{proposition}
\begin{proof}
The $k=1$  case follows from Proposition~\ref{prop:Ra97} below. Now the proposition follows from Theorem~\ref{thm:main} via an induction on both $k$ and $n$, and the following well-known identity: 
\[
	\begin{pmatrix}n-2 \\ k\end{pmatrix}  =  \begin{pmatrix}n-2 \\ k-1\end{pmatrix} + \begin{pmatrix}n-3 \\ k\end{pmatrix}.\qedhere
\]
\end{proof}

\subsection{An example}\label{ex:1}
We shall simplify the notation in examples, and write 
\[
1\dot{3}2 = \CP^J_{s_3, s_1s_3s_2;>0}, \quad\text{etc.}.	
\]

We start with an example for the amplituhedron $\CA_{5,2,2}$. In this case, we construct recursively a collection of cells in $\Gr_{2,5}^{\ge 0}$.

We have $\CC_{4,2,2} = \{2132\}$ and $\CC_{4, 1, 2} = \{3\dot{2}1,21\}$ or $\{\dot{3}21,32\dot{1}\}$. If we start with $\CC_{4, 1, 2} = \{3\dot{2}1,21\}$, then we have 
\begin{align*}
	&\CC_{5,2,2} \\
	= & \iota_{\text{pre}} ( \CC_{4,2,2} ) \sqcup \sigma_{2,5}^{-2} \cdot U^-(s_2s_1)\cdot  \iota_{\text{inc}} ( \CC_{4, 1, 2}) \\
	= &\{2132\} \sqcup \sigma_{2,5}^{-2}  \cdot U^-(s_2s_1)  \cdot  \{4\dot{3}2, 32\}\\
	= &\{2132\} \sqcup \sigma_{2,5}^{-2}  \cdot \{214\dot{3}2, 2132\}\\
	 = &\{2132, 32\dot{1}4\dot{3}2, \dot{3}21\dot{4}32\}.
\end{align*}

On the other hand, we have 
\[
 \CC_{5,2,2} = \{2132, 214\dot{3}2, 3\dot{2}14\dot{3}2\}, \quad \text{ if } \CC_{4, 1, 2}  = \{\dot{3}21,32\dot{1}\}.
\]

\section{Preparation}\label{sec:pre}
In this section, we prove several technical lemmas and study the amplituhedron at $k=1$, which shall be the base case of our induction.

\subsection{Reductions} 
Let $H = I - \{k+m\}$. Let $W_H $ be the parabolic subgroup of $W$, where we denote the longest element by $w_0^H$. 

The image of $Z^t$ via the natural map 
\[
	 \Mat_{n,k+m}^{\times} \rightarrow \Mat_{n,k+m}^{\times} / \GL_{k+m} \cong \Gr_{k+m,n}
\]
is totally positive.  Therefore thanks to \S\ref{subsec:Lu}, we have 
\begin{equation}\label{eq:Z}
	Z =  h \cdot \kbordermatrix{
   		&   1		&	\cdots	&   k+m       	& 	&   \cdots 	&		 n	\\
	& 1 		& \cdots 		& 0		& 0	   & \cdots 	& 0	\\
	& \vdots	&  \ddots 		& \vdots	& \vdots		     &		\ddots	&\vdots\\
	&0 		& \cdots 		& 1 		& 0 		     & \cdots 	& 0\\
} \cdot g
\end{equation}
with a unique $g \in U^+(w_0 w^H_{0}) \text{ and some }  h \in \GL^+_{k+m}$.  Since the morphism $\widetilde{Z}$ is $\GL_{k+m}$-equivariant, we can assume $h \in \GL_{k+m}$ is the identity up to isomorphism. Then we have 
\[
	\widetilde{Z}: \Gr_{k,n}^{\ge 0} \longrightarrow \Gr^{\ge 0}_{k,k+m} \subset \Gr_{k,k+m}.
\]

\begin{lemma}\label{lem:Zg}
Let $C$ be a cell in $\Gr_{k,n}^{\ge 0}$ and $\{C_i\}$ be a (finite) collection cells in $\overline{C}$, such that 
\[
\overline{Z C} = \bigcup_i \overline{Z C_i}, \qquad \text{ for any $Z \in \text{Mat}_{k+m, n}^{>0}$.}
\]

Then  $\overline{Z\cdot { U^-(s_j) \cdot C}} =  \bigcup_{i} \overline{ {Z\cdot U^-(s_j) \cdot C_i}}$.
\end{lemma}

\begin{proof}
 Recall  
\[
	U^-(s_j) \cdot C = \bigcup_{a \in \mathbb{R}_{>0} }y_j(a) \cdot C.
\]
Since $Z \cdot  y_j(a) \in \text{Mat}_{k+m, n}^{>0}$ for any fixed $a \in \mathbb{R}_{>0}$, we have 
\[
	\bigcup_{a \in \mathbb{R}_{> 0}} \overline{Z\cdot y_j(a)  \cdot C} =  \bigcup_{a \in \mathbb{R}_{> 0}}  \bigcup_{i} \overline{{Z\cdot y_j(a) \cdot C_i}}.
\]

Since $\{C_i\}$ is a finite set, the lemma follows by taking the closure on both sides of the equality.
 \end{proof}

\begin{corollary}\label{cor:cyclic}

Let $m$ be even, and let $\CC_{n,k,m}$ be a collection of cells in $\Gr_{k,n}^{\ge 0}$ that triangulates $\CA_{n,k,m}$.  Then $\sigma_{k,n} \cdot \{C_\alpha \}$ triangulates $\CA_{n, k, m}$. 
\end{corollary}

\begin{proof}
Thanks to \eqref{eq:cyclic}, 
\[
\sigma_{k,n}: \text{Mat}_{k+m, n}^{>0} \longrightarrow \text{Mat}_{k+m, n}^{>0}, \quad Z \mapsto Z \cdot \sigma_{k,n}
\]
is a bijection. The corollary follows. 
\end{proof}

\begin{proposition}
\label{prop:reduc}
\begin{enumerate}
	\item 	If $\CC_{n-1,k,m}$ triangulates $\CA_{n-1,k,m}$, then $\iota_{\text{pre}}(\CC_{n-1,k,m})$ triangulates $Z \cdot \overline{\CP^J_{k,n-1}}$.
	\item 	If $\CC_{n-1,k-1,m}$ triangulates $\CA_{n-1,k-1,m}$, then $\iota_{\text{inc}}( \CC_{n-1,k-1,m})$ triangulates $Z \cdot \overline{\CP^J_{k-1,n-1}}$.
\end{enumerate}
\end{proposition}

\begin{remark}
The triangulation of $Z \cdot \overline{\CP^J_{k,n-1}}$ (resp., $Z \cdot \overline{\CP^J_{k-1,n-1}}$) is in the obvious sense generalizing Definition~\ref{def:tri}.
\end{remark}
\begin{proof}
We prove the first claim. The second claim is similar, whose proof shall be omitted. Recall \eqref{eq:Z} we have 
\[
	Z =  \kbordermatrix{
   		&   1		&	\cdots	&   k+m       	& 	&   \cdots 	&		 n	\\
	& 1 		& \cdots 		& 0		& 0	   & \cdots 	& 0	\\
	& \vdots	&  \ddots 		& \vdots	& \vdots		     &		\ddots	&\vdots\\
	&0 		& \cdots 		& 1 		& 0 		     & \cdots 	& 0\\
} \cdot g,
\]
where $g \in U^+(w_0 w^H_{0})$.

Let $\GL_{n-1}$ be the standard Levi subgroup of $\GL_n$ with simple roots $I - \{ n-1\}$. Let $U^+_{n-1,\ge 0} = U^+_{\ge 0} \cap \GL_{n-1}$. Let $\pi$ be the projection $\pi: U^+_{\ge 0} \rightarrow U^+_{n-1,\ge 0}$.  We write 
\[
\pi(Z) = 
  \kbordermatrix{
   		&   1		&	\cdots	&   k+m       	 	&   \cdots 	&		 n-1	\\
	& 1 		& \cdots 		& 0			   & \cdots 	& 0	\\
	& \vdots	&  \ddots 		& \vdots			     &		\ddots	&\vdots\\
	&0 		& \cdots 		& 1 		 		     & \cdots 	& 0\\
} \cdot \kbordermatrix{ & 1 & \cdots & n-1 & n\\
 & 1 & \cdots & 0 & 0\\
 & \vdots & \ddots & \vdots & \vdots\\
 & 0 & \cdots & 1 & 0
}
\cdot
\pi(g) \in \Mat_{k+m, n-1}^{>0},
\]
Note that the image of the $k$-preserving embedding $\Gr^{\ge 0}_{k,n-1} \to \Gr^{\ge 0}_{k,n} $ is $U^+_{\ge 0}$-invariant and stabilized  by $x_1(a)$ for any $a \in \mathbb{R}$.
Hence we have the following commutative diagram 
\[
\xymatrix{\Gr^{\ge 0}_{k,n-1} \ar[dr]_-{\widetilde{\pi(Z)}} \ar[r] & \Gr^{\ge 0}_{k,n} \ar[d]^-{\tilde{Z}} \\
& \Gr^{\ge 0}_{k,k+2}}.
\]
The claim follows.
\end{proof}

\subsection{The cyclic polytope}
The $k=1$ amplituhedrons (for any $m$) are also called cyclic polytopes. They have been studied in details by Rambau \cite{Ra97}.  

For $1<a < n$, we set 
$$
\CP^J_{1, n}(a)=\CP^J_{s_{[a-1, 2]} , s_{[a, 1]} ; >0} \subset \Gr_{1,n}^{\ge 0}.
$$ 

\begin{proposition}\cite{Ra97}\label{prop:Ra}
The collection of  cells 
\[
	\{ \CP^J_{1, n}(a) \vert 1<a < n\}
\]
triangulates the amplituhedron $\CA_{n,1,2}$ for any $n\ge 3$.
\end{proposition}

\begin{proposition}\label{prop:Ra97}  We have $\overline{Z \cdot \CP^J_{1,n-1}} \cup  \overline{ Z \cdot \sigma_{1,n}^{-2} \cdot U^-(s_2s_1) \cdot \CP^J_{1,1;>0}} = \CA_{n, 1, 2} $.

\end{proposition}

\begin{proof}
Thanks to  Proposition~\ref{prop:Ra}, we have 
\[
	\CA_{n, 1, 2} = \left(\bigcup_{a =2 }^{n-2}\overline{Z \cdot\CP^J_{1, n}(a) } \right)\bigcup \overline{Z \cdot\CP^J_{1, n}(n-1)}.
\]
By Proposition~\ref{prop:reduc} and Proposition~\ref{prop:Ra}, we have 
\[
	\overline{ Z \cdot \CP^J_{1,n-1}} = \bigcup_{a=2}^{n-2}\overline{Z \cdot  \CP^J_{1, n}(a)}.
\]
It follows from Proposition~\ref{prop:cyc} that 
\[
\sigma_{1,n}^{-2} \cdot U^-(s_2s_1) \cdot \CP^J_{1,1;>0} = \CP^J_{1, n}(n-1).
\]
The proposition follows.
\end{proof}

\begin{remark}
We only need Proposition~\ref{prop:Ra97} in order to prove Theorem~\ref{thm:main}. One can actually deduce Proposition~\ref{prop:Ra97} directly using linear algebra. We refer to \cite[\S3]{AHT13} fo some discussion. 
\end{remark}

\section{Proof of the main theorem}\label{sec:proof}
We prove Theorem~\ref{thm:main} in this section via induction on both $k$ and $n$.  We shall often consider various different initial data, in particular when applying the induction hypothesis.
\subsection{Injectivity}\label{subsec:inj}

\begin{lemma}\label{lem:inj}

\begin{enumerate}
	\item Let $C \in \iota_{\text{pre}} ( \CC_{n-1, k, 2} )$. Then the map
	\[
	\widetilde{Z} : C \longrightarrow \Gr_{k,k+2}
\]
is injective. 
		
	\item 
Let $C \in \iota_{\text{inc}} (\CC_{n-1, k-1, 2} )$. Then the map
\[
	\widetilde{Z} : \sigma_{k,n}^{-2} \cdot U^-(s_2s_1) \cdot C \longrightarrow \Gr_{k,k+2}
\]
is injective. 
\end{enumerate}
\end{lemma}

\begin{proof}
Note that item (1) follows from  the induction hypothesis and Proposition~\ref{prop:reduc}. We prove (2) here. 
  
Since $Z \in \text{Mat}_{k+m, n}^{> 0}$, then
\[
	Z \cdot \sigma_{k,n}^{-2} \cdot y_2(a_2) y_1(a_1) \in  \text{Mat}_{k+m, n}^{> 0}, \quad \text{ for any } a_2, a_1 \in \mathbb{R}_{>0}.
\]
It follows from the induction hypothesis and Proposition~\ref{prop:reduc} that  
\[
\widetilde{Z} : \sigma_{k,n}^{-2} \cdot y_2(a_2) y_1(a_1) \cdot C \longrightarrow \Gr_{k,k+2}\quad
\]
is injective for fixed $a_1, a_2  \in \mathbb{R}_{>0}$.

Hence it suffices to prove 
\[
Z \cdot \sigma_{k,n}^{-2} \cdot y_2(a_2) y_1(a_1) \cdot C \cap Z \cdot \sigma_{k,n}^{-2} \cdot y_2(a'_2) y_1(a'_1) \cdot C = \emptyset, \quad \text{ if } (a_1, a_2) \neq (a'_1, a'_2).
\]
Let  $Z' =  Z \cdot \sigma_{k,n}^{-2} \in \Mat_{k+m, n}^{>0}$.
Since $\iota_{\text{inc}} ( \CC_{n-1, k-1, 2} )$ triangulates $Z' \cdot \CP^J_{k-1,n-1}$ by induction, it suffices to prove 
\[
Z'  \cdot y_2(a_2) y_1(a_1) \cdot  \CP^J_{k-1,n-1} \cap Z' \cdot y_2(a'_2) y_1(a'_1) \cdot \CP^J_{k-1,n-1} = \emptyset, \quad \text{ if } (a_1, a_2) \neq (a'_1, a'_2).
\]

Recall again \eqref{eq:Z}, we can assume 
\[
Z' =  \kbordermatrix{
   		&   1		&	\cdots	&   k+m       	& 	&   \cdots 	&		 n	\\
	& 1 		& \cdots 		& 0		& 0	   & \cdots 	& 0	\\
	& \vdots	&  \ddots 		& \vdots	& \vdots		     &		\ddots	&\vdots\\
	&0 		& \cdots 		& 1 		& 0 		     & \cdots 	& 0\\
} \cdot g
\]
with some $g \in U^+(w_0 w^H_{0})$.

Thanks to \cite[\S1.3]{Lu94}, we have 
\[
	g \cdot y_2(a_2) y_1(a_1) = y_2(b_2) y_1(b_1) \cdot \chi \cdot g',
\]
for some diagonal matrix $\chi \in G$, $g' \in U^+(w_0 w^H_{0})$, and $b_1$, $b_2 \in \mathbb{R}_{> 0}$. In particular, the map $(a_2, a_1) \mapsto (b_2, b_1)$ is bijective.

Also note that 
\begin{align*}
 \kbordermatrix{
   		&   1		&	\cdots	&   k+m       	& 	&   \cdots 	&		 n	\\
	& 1 		& \cdots 		& 0		& 0	   & \cdots 	& 0	\\
	& \vdots	&  \ddots 		& \vdots	& \vdots		     &		\ddots	&\vdots\\
	&0 		& \cdots 		& 1 		& 0 		     & \cdots 	& 0\\
}  \cdot y_2(a_2) y_1(a_1) \\
= y_2(a_2) y_1(a_1) \cdot  \kbordermatrix{
   		&   1		&	\cdots	&   k+m       	& 	&   \cdots 	&		 n	\\
	& 1 		& \cdots 		& 0		& 0	   & \cdots 	& 0	\\
	& \vdots	&  \ddots 		& \vdots	& \vdots		     &		\ddots	&\vdots\\
	&0 		& \cdots 		& 1 		& 0 		     & \cdots 	& 0\\
}.
\end{align*}

Therefore, we have 
\[
	Z'  \cdot y_2(a_2) y_1(a_1) \cdot  \CP^J_{k-1,n-1}  \subset  y_2(b_2) y_1(b_1) \CP^J_{1, s_{[3,2]}s_{[4,3]}\cdots s_{[k+1,k]}} \subset \Gr_{k,k+2}^{\ge 0}
\]
and 
\[
Z'  \cdot y_2(a'_2) y_1(a'_1) \cdot  \CP^J_{k-1,n-1}  \subset  y_2(b'_2) y_1(b'_1) \CP^J_{1, s_{[3,2]}s_{[4,3]}\cdots s_{[k+1,k]}} \subset \Gr_{k,k+2}^{\ge 0},
\]
where $(b_1, b_2) \neq (b'_1, b'_2)$, if $(a_1, a_2) \neq (a'_1, a'_2)$. The lemma follows.
\end{proof}

\begin{corollary}\label{cor:open}
Let $C \in  \CC_{n,k,m} $. Then $Z\cdot C$ is open in $\Gr_{k,k+m}$.
\end{corollary}

\begin{proof}
Recall \eqref{eq:Z}, we assume 
\[
Z =   \kbordermatrix{
   		&   1		&	\cdots	&   k+m       	& 	&   \cdots 	&		 n	\\
	& 1 		& \cdots 		& 0		& 0	   & \cdots 	& 0	\\
	& \vdots	&  \ddots 		& \vdots	& \vdots		     &		\ddots	&\vdots\\
	&0 		& \cdots 		& 1 		& 0 		     & \cdots 	& 0\\
} \cdot g,
\]
with a unique $g \in U^+(w_0 w^H_{0}) \text{ and some }  h \in \GL^+_{k+m}$. Then we have 
\[
	\tilde{Z}: \Gr_{k,n}^{\ge 0} \supset C \cong \mathbb{R}_{>0}^{km} \longrightarrow \mathbb{R}^{km}_{>0 } \cong \Gr^{> 0}_{k,k+m} \subset \Gr_{k,k+m}.
\]

We have already prove that $\tilde{Z}$ is injective in Lemma~\ref{lem:inj}. Hence by Brouwer's theorem of invariance of the domain, we know the image is open.
\end{proof}

\subsection{Disjointness}
Note that the disjointness within each family  $ \iota_{\text{pre}} (  \CC_{n-1, k, 2} )$ and $ \sigma_{k,n}^{-2} \cdot U^-(s_2s_1) \cdot \iota_{\text{inc}} ( \CC_{n-1, k-1, 2} )$ follows by induction hypothesis and Proposition~\ref{prop:reduc}. Therefore, it suffices to prove, for any $C \in \ip$ and $D \in \ii$, 
\[
	Z \cdot C  \cap Z \cdot D =\emptyset.
	\]

Let $Z = [Z_1,Z_2,\dots,Z_n]$, where $Z_i$'s are the columns of $Z$. For any $1 \le i, j \le n$, we define the following map
\begin{align*}
	[i,j]: \Gr_{k,k+2}^{\ge 0} &\longrightarrow \mathbb{R}/\GL_1^{+} = \{-1, 0 ,1\},\\
		A \in \text{Mat}^{\ge 0}_{k+2 \times k}&\mapsto \det{[A,Z_i,Z_j]}/\GL_1^{+},
\end{align*}
where $[A,Z_i,Z_j] \in \Mat_{k+2, k+2}$ denote the new matrix constructed in the obvious way. Note that while the matrix $[A,Z_i,Z_j] \in \Mat_{k+2, k+2}$ depends on the choice of $A$, the function $[i,j]$ depends only on the point in $\Gr_{k,n}^{\ge 0}$ we start with.

We also defined the pullback 
\[
	\widetilde{[i,j]}:\xymatrix{ \Gr_{k,n}^{\ge 0} \ar[r]^-{\tilde{Z}}  & \Gr_{k,k+2}^{\ge 0} \ar[r]^-{[i,j]} & \mathbb{R}/\GL_1^{+}}.
\] 

\begin{lem}
Let $C \in \ip$ and $D \in \ii$. We have 
\begin{align*}
	[1,n-1](Z\cdot C) = \widetilde{[1,n-1]} (C) &=(-1)^{k}, \\
	[1,n-1](Z\cdot D)  =  \widetilde{[1,n-1]} ( D) &= (-1)^{k-1}.
\end{align*}

Therefore, we have 
\[
	Z \cdot C \cap Z \cdot D=\emptyset.
\]
\end{lem}

\begin{proof}
Thanks to Proposition~\ref{prop:reduc}, we know 
\[
\ip \text{ triangulates } Z \cdot \overline{\CP^J_{k,n-1}},
\]
\[
\ii \text{ triangulates } Z \cdot \overline{\CP^J_{s_{[n-2, k+1]}, s_{[n-k, 1]} s_{[n-k+1, 2]} \cdots s_{[n-1, k]}; >0}}.
\]

It follows by direct computation (mod $\GL^{+}_{1}$) that 
\begin{align*}
&\widetilde{[1,n-1]} (\CP^J_{k,n-1}) \\
= & \sum_{1\le i_1 < i_2 \cdots <i_k\le n-1} \Delta_{i_1,i_2,\cdots, i_k }(\CP^J_{k,n-1})\det[Z_{i_1}, Z_{i_2}, \cdots, Z_{i_{k}}, Z_1, Z_{n-1}] \\
= & (-1)^{k}.
\end{align*}

We similarly have 
\begin{align*}
&\widetilde{[1,n-1]}  (\CP^J_{s_{[n-2, k+1]}, s_{[n-k, 1]} s_{[n-k+1, 2]} \cdots s_{[n-1, k]}; >0}) \\
=& \sum_{1\le i_1 < i_2 \cdots <i_k= n}\Delta_{i_1,i_2,\cdots, i_k }(\CP^J_{s_{[n-2, k+1]}, s_{[n-k, 1]} s_{[n-k+1, 2]} \cdots s_{[n-1, k]}; >0})\\
 &\cdot \det[Z_{i_1}, Z_{i_2}, \cdots, Z_{n}, Z_1, Z_{n-1}]\\
 = & (-1)^{k-1}.
\end{align*}
So we have $\widetilde{[1,n-1]}  (C) \subset \{0, (-1)^{k}\}$ and $\widetilde{[1,n-1]} (D) \subset \{0,(-1)^{k-1}\}$. It also follows directly from the computation that we can only have $\widetilde{[1,n-1]}  (C) =$ either $0$ or $(-1)^{k}$. Similarly, we can only have  $\widetilde{[1,n-1]} (D) = $ either $0$ or $(-1)^{k-1}$. 

Then thanks to Corollary \ref{cor:open}, both $Z\cdot C$ and $Z \cdot D$ are open in $\Gr_{k,k+2}^{\ge 0}$. Therefore, we must have 
\[
\widetilde{[1,n-1]}  (C)   = [1,n-1](Z\cdot C) =    (-1)^{k},  \widetilde{[1,n-1]} (D) = [1,n-1](Z\cdot D) = (-1)^{k-1}.
\]
The lemma follows. 
\end{proof}

\subsection{Surjectivity}

\begin{lemma}\label{lem:k=1}We have 
\[
	\overline{Z \cdot \CP^J_{s_{[n-2, k+1]}, s_{[n-k, 1]} s_{[n-k+1, 2]} \cdots s_{[n-1, k]}; >0}} \cup \overline{Z \cdot \CP^J_{1, s_{[n-k-1, 1]} s_{[n-k, 2]} \cdots s_{[n-2, k]}; >0}} = \mathcal{A}_{n,k,2}.
\]
\end{lemma}

\begin{proof}

The open cell in $\Gr_{k, n}^{\ge 0}$ is 
\[
\CP^J_{1, s_{[n-k, 1]} s_{[n-k+1, 2]} \cdots s_{[n-1, k]}; >0}=U^{-}(s_{[n-k, 1]} s_{[n-k+1, 2]}\cdots s_{[n-2, k-1]}) \cdot \CP^J_{1, s_{[n-1, k]}; >0}.
\]

Note that 
\begin{align*}
\iota^k_{\text{inc}} (\CP^J_{1, s_{[n-k, 1]}; >0}) &= \CP^J_{1, s_{[n-1, k]}; >0},\\
\iota^k_{\text{inc}} (\CP^J_{1, s_{[n-k-1, 1]}; >0}) &= \CP^J_{1, s_{[n-2, k]}; >0},\\
\iota^k_{\text{inc}} (\CP^J_{s_{[n-k-1,2]}, s_{[n-k, 1]}; >0}) &=\CP^J_{s_{[n-2, k+1]}, s_{[n-1, k]}; >0}.
\end{align*}

By Proposition~\ref{prop:Ra97}, we have 
\[
 \overline{Z \cdot \CP^J_{1, s_{[n-k, 1]}; >0}} = \overline{Z \cdot  \CP^J_{1, s_{[n-k-1, 1]}; >0}} \cup \overline{Z \cdot  \CP^J_{s_{[n-k-1, 2]}, s_{[n-k, 1]}; >0}}.
\]

Then thanks to Proposition~\ref{prop:reduc}, we know 
\[
 \overline{Z \cdot \CP^J_{1, s_{[n-1, k]}; >0}} = \overline{Z \cdot  \CP^J_{1, s_{[n-2, k]}; >0}} \cup \overline{Z \cdot  \CP^J_{s_{[n-2, k+1]}, s_{[n-1, k]}; >0}}.
\]

Then by Lemma~\ref{lem:Zg}, 
$\mathcal{A}_{n,k,2}$ is covered by the closure of the union of  the images of
\begin{align*}
U^{-}(s_{[n-k, 1]} s_{[n-k+1, 2]}\cdots s_{[n-2, k-1]}) \cdot \CP^J_{s_{[n-2, k+1]}, s_{[n-1, k]}; >0} \\
=\CP^J_{s_{[n-2, k+1]}, s_{[n-k, 1]} s_{[n-k+1, 2]} \cdots s_{[n-1, k]}; >0},
\end{align*}
and  
$$
U^{-}(s_{[n-k, 1]} s_{[n-k+1, 2]}\cdots s_{[n-2, k-1]}) \cdot \CP^J_{1, s_{[n-2, k]}; >0}=\CP^J_{1, s_{[n-k-1, 1]} s_{[n-k, 2]} \cdots s_{[n-2, k]}; >0}.
$$
The Lemma is proved. 
\end{proof}

\begin{proposition}We have 

\[
	\bigcup_{C \in \CC_{n,k,2}} \overline{Z \cdot C}  = \CA_{n,k,2}.
\]
\end{proposition}

\begin{proof}
Thanks to Proposition~\ref{prop:reduc}, we have 
\[
	\bigcup_{C \in \ip} \overline{Z \cdot C} = \overline{Z \cdot {\CP^J_{s_{[n-2, k+1]}, s_{[n-k, 1]} s_{[n-k+1, 2]} \cdots s_{[n-1, k]}; >0}}}.
\]

Thanks to Proposition~\ref{prop:reduc} and Proposition~\ref{prop:cyc}, we have 
\[
	\bigcup_{C \in \ii} \mkern-60mu \overline{Z \cdot C}  = \overline{Z \cdot \CP^J_{s_{[n-2, k+1]}, s_{[n-1, k]}; >0}}.
\]
Now the proposition follows from Lemma~\ref{lem:k=1}.
\end{proof}

\section{An explicit collection of cells}\label{4-explicit}

In this section we construct explicitly one (non-recursive) collection of cells in $\Gr_{k,n}^{\ge 0}$ that triangulates $\CA_{n,k,2}$. 

For $1<a_1<a_2<\ldots<a_k \le n-1$, we set 
$$
\CP^J_{k, n}(a_1, a_2, \ldots, a_k)=\CP^J_{s_{[a_1-1, 2]} s_{[a_2-1, 3]} \cdots s_{[a_k-1, k+1]}, s_{[a_1, 1]} s_{[a_2, 2]} \cdots s_{[a_k, k]}; >0} \subset \Gr_{k,n}^{\ge 0}.
$$ 
Note that  $\dim \CP^J_{k, n}(a_1, \ldots, a_k)=2 k=\dim \Gr_{k,k+2}^{\ge 0}$.

\begin{theorem}\label{thm:m=2} 
The cells 
$$
\{\CP^J_{k, n}(a_1, a_2, \ldots, a_k); 1<a_1<a_2<\ldots<a_k \le n-1\}
$$
 gives a triangulation of the amplituhedron $\mathcal{A}_{n,k,2} $.
\end{theorem}

\begin{example}As in Example~\ref{ex:1}, we simplify the notation and write 
\[
1\dot{3}2 = \CP^J_{s_3, s_1s_3s_2; >0}, \quad\text{etc.}.	
\]
	\begin{enumerate}
		\item		In case of $\mathcal{A}_{4,1,2}$, the collection of cells in $\Gr_{1,4}^{\ge 0}$ is the following: 
		\[	21, \quad 3\dot{2}1.
		\]	
		In this case, the theorem follows from Proposition~\ref{prop:Ra97}.
		\item		In the case of $\mathcal{A}_{5,2,2}$, the collection of cells in $\Gr_{2,5}^{\ge 0}$ is the following:
		\[
			2132, \quad 214\dot{3}2, \quad 3\dot{2}14\dot{3}2.
		\]
		\item		In the case of $\mathcal{A}_{6,2,2}$, the collection of cells in $\Gr_{2,6}^{\ge 0}$ is the following:
		\[
			2132, \quad 214\dot{3}2, \quad  3\dot{2}14\dot{3}2,  \quad 4\dot{3}\dot{2}15\dot{4}\dot{3}2,  \quad 3\dot{2}15\dot{4}\dot{3}2,  \quad 215\dot{4}\dot{3}2.
		\]
		\item		In the case of $\mathcal{A}_{7,3,2}$, the collection of cells in $\Gr_{3,7}^{\ge 0}$ is the following:
		\begin{align*}	
			213243, \quad 21325\dot{4}3, \quad 21326\dot{5}\dot{4}3,  \quad  214\dot{3}25\dot{4}3,  \quad  214\dot{3}26\dot{5}\dot{4}3,  \quad  215\dot{4}\dot{3}26\dot{5}\dot{4}3, 
			\\
			3\dot{2}1 4\dot{3}25\dot{4}3, \quad 3\dot{2}1 4\dot{3}26\dot{5}\dot{4}3, \quad 3\dot{2}1 5\dot{4}\dot{3}2 6\dot{5}\dot{4}3,  \quad 4\dot{3}\dot{2}1 5\dot{4}\dot{3}2 6\dot{5}\dot{4}3.
		\end{align*}
\end{enumerate}
\end{example}

\begin{proof}
We prove the theorem by induction again. The base case $k=1$ is Proposition~\ref{prop:Ra97}. 
Now let 
\[
\CC_{n-1,k,2} = \{\CP^J_{k, n-1}(a_1, a_2, \ldots, a_{k} ); 1<a_1<a_2<\ldots<a_{k} \le n-2\}
\]
and 
\[
 \CC_{n-1,k-1,2}  = \{\sigma^{-1}_{k-1,n-1}\cdot \CP^J_{k-1, n-1}(a_1, a_2, \ldots, a_{k-1} ); 1<a_1<a_2<\ldots<a_{k-1} \le n-2\}.
\]

We know that 
\[
 \CC_{n-1,k,2}   \text{ triangulates } \CA_{n-1,k,2},
\]
and 
\[
 \CC_{n-1,k-1,2} \text{ triangulates } \CA_{n-1,k-1,2}, \text{ thanks to Corollary~\ref{cor:cyclic}}.
\]
We immediately have 
\[
\ip  = \{\CP^J_{k, n}(a_1, a_2, \ldots, a_{k} ); 1<a_1<a_2<\ldots<a_{k} \le n-2\}.
\]
It follows from direct computation that 
\begin{align*}
	&\sigma_{k,n}^{-2} \cdot U^-(s_2s_1) \cdot \iota _{\text{inc}}\Big(\sigma^{-1}_{k-1,n-1} \cdot \CP^J_{k-1, n-1}(a_1, a_2, \ldots, a_{k-1} )\Big) \\
	=&  \CP^J_{k-1, n-1}(a_1, a_2, \ldots, a_{k-1}, a_k=n-1), \text{ with } 1 < a_1 < \dots < a_{k-1}\le n-2.
\end{align*} 
Therefore 
\begin{align*}
&\ii \\
= & \{\CP^J_{k, n}(a_1, a_2, \ldots, a_{k} ); 1<a_1<\ldots<a_{k-1}<a_{k}= n-1\}.
\end{align*}

Finally, thanks to Theorem~\ref{thm:main}, 
\begin{align*}
 \CC_{n,k,2}  &=\ip \sqcup \ii \\ 
 &= \{\CP^J_{k, n}(a_1, a_2, \ldots, a_k); 1<a_1<a_2<\ldots<a_k \le n-1\}
\end{align*}
triangulates $\CA_{n,k,2}$.
\end{proof}

\section{Diagrammatic expression}\label{5-diag}
The cells in the totally nonnegative Grassmannian $\Gr_{k,n}^{\ge 0}$ can also be parameterized by plabic graphs, $\Le$-diagrams, decorated permutations, and many other combinatorial models (\cite{Pos}). The plabic graphs are called on-shell diagrams in the physical context, in which language the original BCFW recursion (for the $m=4$ amplituhedron) was formulated. We refer to \cite{KWZ} and  \cite{Pos} for details on different combinatorial parametrization of cells. 

In this section, we translate our main theorem into plabic graphs. Since the result in this section is not needed in the other part of the paper, we shall keep the section concise. We remark that we follow the reformulation of the BCFW recursion in \cite{BH15}. Therefore there is NO shift versus \cite[\S4 \& \S5]{KWZ}.

\begin{example}Let us first explain the translation between plabic graphs and Bruhat intervals vis an example. Details can be found in \cite[\S19 \& \S20]{Pos}.
We consider the plabic graph

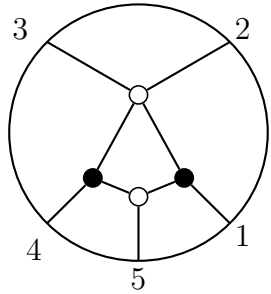
\[
\begin{tikzpicture}[baseline=(current bounding box.center)]

\tikzstyle{out1}=[inner sep=0,minimum size=2.4mm,circle,draw=black,fill=black,semithick]
\tikzstyle{in1}=[inner sep=0,minimum size=2.4mm,circle,draw=black,fill=white,semithick]

\pgfmathsetmacro{\radius}{1.7};
\pgfmathsetmacro{\radiuss}{.5};
\pgfmathsetmacro{\shift}{.5};

\draw[thick](0,0)circle[radius=\radius];

\node[in1](s0)at(90:\shift){};

\node[inner sep=0](b2)at(135-1*90:\radius){};
\node at(135-1*90:\radius+.24){$2$};

\node[inner sep=0](b1)at(135-2*90:\radius){};
\node at(135-2*90:\radius+.24){$1$};

\node[inner sep=0](bn)at(-90:\radius){};
\node at(-90:\radius+.24){$5$};

\node[inner sep=0](bn-1)at(-135:\radius){};
\node[yshift=-1ex] at(-135:\radius+.24){\Large$\scriptstyle 4$};

\node[inner sep=0](bn-2)at(135:\radius){};
\node[xshift=-1ex] at(135:\radius+.24){\Large$\scriptstyle 3$};

\node[out1](i1)at(-45:\radius/2){};
\node[out1](in-1)at(-135:\radius/2){};
\node[in1](in)at(-90: \radius/2){};

\path[thick](b1.center)edge(i1) (b2.center)edge(s0) (in-1)edge(bn-1.center) (in)edge(bn.center) (s0)edge(bn-2.center) (i1)edge(in) (in)edge(in-1) (i1)edge(s0) (in-1)edge(s0);
\end{tikzpicture}
\]

We have the corresponding decorated permutation:
\[\pi = \begin{matrix}
	1 & 2 & 3 & 4 & 5 \\
	4 & 5 & 2 & 1 & 3
\end{matrix}\quad,
\]

with the set of anti-excedance
\[
	I(\pi) = \{1, 2 ,3\}.
\]
We then have the associated pipe dream and  $\Le$-diagram:
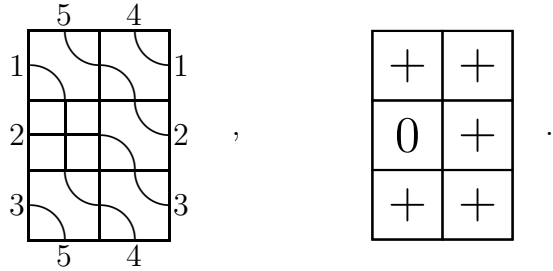
\vspace{.2cm}
\[
\begin{tikzpicture}[baseline=(current bounding box.center)]
\pgfmathsetmacro{\unit}{0.922};
\useasboundingbox(0,0)rectangle(2*\unit,-3*\unit);
\coordinate (vstep)at(0,-0.24*\unit);
\coordinate (hstep)at(0.17*\unit,0);
\draw[thick](0,0)--(2*\unit,0) (0,0)--(0,-3*\unit);
\node[inner sep=0]at(0,0){\scalebox{1.6}{\begin{ytableau}
\none \\
\none \\
\none \\
\none \\
\none & \none & \none & \none & \none & \elbow & \elbow & \none &\none &\none \\
\none & \none & \none & \none & \none & \cross & \elbow  & \none & \none & \none\\
\none & \none & \none & \none & \none &  \elbow &  \elbow & \none \\
\none & \none & \none & \none & \none & \none & \none
\end{ytableau}}};

\node[inner sep=0]at($(2*\unit,-0.5*\unit)+(hstep)$){$1$};
\node[inner sep=0]at($(2*\unit,-1.5*\unit)+(hstep)$){$2$};
\node[inner sep=0]at($(2*\unit,-2.5*\unit)+(hstep)$){$3$};
\node[inner sep=0]at($(1.5*\unit,-3*\unit)+(vstep)$){$4$};
\node[inner sep=0]at($(0.5*\unit,-3*\unit)+(vstep)$){$5$};

\node[inner sep=0]at($(1.5*\unit,0)-(vstep)$){$4$};
\node[inner sep=0]at($(0.5*\unit,0)-(vstep)$){$5$};
\node[inner sep=0]at($(0,-0.5*\unit)-(hstep)$){$1$};
\node[inner sep=0]at($(0,-1.5*\unit)-(hstep)$){$2$};
\node[inner sep=0]at($(0,-2.5*\unit)-(hstep)$){$3$};

\end{tikzpicture}
\qquad ,\qquad \qquad 
\begin{tikzpicture}[baseline=(current bounding box.center)]
\pgfmathsetmacro{\scalar}{1.6};
\pgfmathsetmacro{\unit}{\scalar*0.922/1.6};
\draw[thick](0,0)rectangle(2*\unit,-3*\unit);
\foreach \x in {1,...,2}{
\draw[thick](\x*\unit-\unit,0)rectangle(\x*\unit,-\unit);}
\foreach \x in {1,...,2}{
\draw[thick](\x*\unit-\unit,-\unit)rectangle(\x*\unit,-2*\unit);}
\foreach \x in {1,...,2}{
\draw[thick](\x*\unit-\unit,-2*\unit)rectangle(\x*\unit,-3*\unit);}
\node[inner sep=0]at(0.5*\unit,-0.5*\unit){\scalebox{\scalar}{$+$}};
\node[inner sep=0]at(1.5*\unit,-0.5*\unit){\scalebox{\scalar}{$+$}};
\node[inner sep=0]at(0.5*\unit,-1.5*\unit){\scalebox{\scalar}{$0$}};
\node[inner sep=0]at(1.5*\unit,-1.5*\unit){\scalebox{\scalar}{$+$}};
\node[inner sep=0]at(0.5*\unit,-2.5*\unit){\scalebox{\scalar}{$+$}};
\node[inner sep=0]at(1.5*\unit,-2.5*\unit){\scalebox{\scalar}{$+$}};

\end{tikzpicture}\quad .\]
\vspace{.2cm}

Finally, we read the Marsh-Rietsch parametrization from the $\Le$-diagram: 

\[
	\begin{tikzpicture}[baseline=(current bounding box.center)]
\pgfmathsetmacro{\scalar}{1.6};
\pgfmathsetmacro{\unit}{\scalar*0.922/1.6};
\draw[thick](0,0)rectangle(2*\unit,-3*\unit);
\foreach \x in {1,...,2}{
\draw[thick](\x*\unit-\unit,0)rectangle(\x*\unit,-\unit);}
\foreach \x in {1,...,2}{
\draw[thick](\x*\unit-\unit,-\unit)rectangle(\x*\unit,-2*\unit);}
\foreach \x in {1,...,2}{
\draw[thick](\x*\unit-\unit,-2*\unit)rectangle(\x*\unit,-3*\unit);}
\node[inner sep=0]at(0.5*\unit,-0.5*\unit){\scalebox{\scalar}{$2$}};
\node[inner sep=0]at(1.5*\unit,-0.5*\unit){\scalebox{\scalar}{$1$}};
\node[inner sep=0]at(0.5*\unit,-1.5*\unit){\scalebox{\scalar}{$\dot{3}$}};
\node[inner sep=0]at(1.5*\unit,-1.5*\unit){\scalebox{\scalar}{$2$}};
\node[inner sep=0]at(0.5*\unit,-2.5*\unit){\scalebox{\scalar}{$4$}};
\node[inner sep=0]at(1.5*\unit,-2.5*\unit){\scalebox{\scalar}{$3$}};

\end{tikzpicture}=
3214\dot{3}2 = \CP^J_{s_3, s_{3} s_{2} s_{1} s_{4} s_{3} s_{2}; >0}.
\]
\vspace{.2cm}
\end{example}

\begin{proposition} \label{prop:diagram}
The plabic graphs/on-shell diagram  expansion of  Theorem~\ref{thm:main} can be represented as follows: 
$$
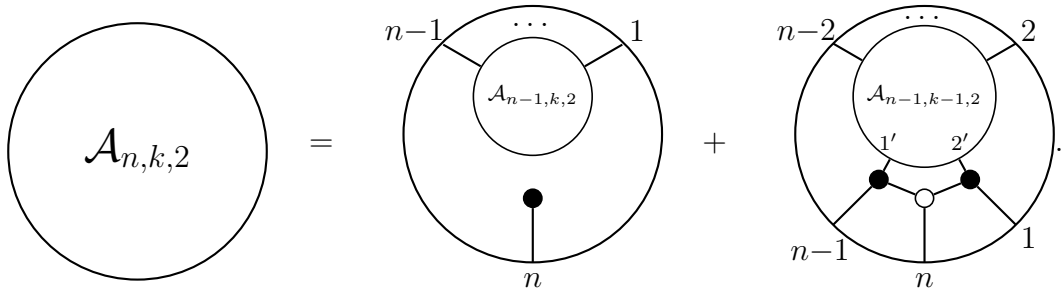
\quad\begin{tikzpicture}[baseline=(current bounding box.center)]
\pgfmathsetmacro{\radius}{1.7};
\draw[thick](0,0)circle[radius=\radius];
\node[]at (0,0) {\Large $\CA_{\scriptstyle n,k,2}$};
\end{tikzpicture}
\quad
=
\quad
\begin{tikzpicture}[baseline=(current bounding box.center)]
\pgfmathsetmacro{\radius}{1.7};
\tikzstyle{out1}=[inner sep=0,minimum size=2.4mm,circle,draw=black,fill=black,semithick]
\tikzstyle{in1}=[inner sep=0,minimum size=2.4mm,circle,draw=black,fill=white,semithick]
\draw[thick](0,0)circle[radius=\radius];

\node[circle,draw=black,fill=white,semithick] (Ak) at (0,0.5) {$\scriptstyle \CA_{ n-1,k,2}$};

\node[inner sep=0](bn-1)at(135:\radius){};
\node[xshift=-1ex] at(135:\radius+.24){\Large$\scriptstyle n-1$};

\node[inner sep=0](b1)at(135-1*90:\radius){};
\node at(135-1*90:\radius+.24){$1$};

\node[out1](in)at(-90: \radius/2){};
\node[inner sep=0](bn)at(-90:\radius){};
\node at(-90:\radius+.24){$n$};
\node[thick] at(90:\radius/2+.6){$\dots$};
\path[thick](bn.center)edge(in) (bn-1)edge(Ak) (b1)edge(Ak);
\end{tikzpicture}
\quad
+ 
\quad\begin{tikzpicture}[baseline=(current bounding box.center)]

\tikzstyle{out1}=[inner sep=0,minimum size=2.4mm,circle,draw=black,fill=black,semithick]
\tikzstyle{in1}=[inner sep=0,minimum size=2.4mm,circle,draw=black,fill=white,semithick]

\pgfmathsetmacro{\radius}{1.7};
\draw[thick](0,0)circle[radius=\radius];

\node[circle,draw=black,fill=white,semithick] (Ak-1) at (0,0.5) {$\scriptstyle \CA_{ n-1,k-1,2}$};

\node[inner sep=0](b2)at(135-1*90:\radius){};
\node at(135-1*90:\radius+.24){$2$};

\node[inner sep=0](b1)at(135-2*90:\radius){};
\node at(135-2*90:\radius+.24){$1$};

\node[inner sep=0](bn)at(-90:\radius){};
\node at(-90:\radius+.24){$n$};

\node[inner sep=0](bn-1)at(-135:\radius){};
\node[yshift=-1ex] at(-135:\radius+.24){\Large$\scriptstyle n-1$};

\node[inner sep=0](bn-2)at(135:\radius){};
\node[xshift=-1ex] at(135:\radius+.24){\Large$\scriptstyle n-2$};

\node[out1](i1)at(-45:\radius/2){};
\node[out1](in-1)at(-135:\radius/2){};
\node[in1](in)at(-90: \radius/2){};

\path[thick](b1.center)edge(i1) (b2.center)edge(Ak-1) (in-1)edge(bn-1.center) (in)edge(bn.center) (Ak-1)edge(bn-2.center) (i1)edge(in) (in)edge(in-1) (i1)edge node[above, pos=0.8] {$\scriptstyle 2'$} (Ak-1) (in-1)edge node[above, pos=0.8] {$\scriptstyle 1'$} (Ak-1);

\node[thick] at(90:\radius/2+.7){$\dots$};
\end{tikzpicture}.
$$
\end{proposition}

\begin{remark}
Note that the labels of the two internal edges are irrelevant to the triangulation, thanks to the $\sigma_{k-1,n-1}$-symmetry of the triangulation of $\CA_{n-1,k-1,2}$ in Corollary~\ref{cor:cyclic}. However, different labeling results in different collection of cells in $\CC_{n,k,2}$.
Here the labels $1'$ and $2'$ match the formulation as stated in Theorem~\ref{thm:main}.
\end{remark}

The diagrammatic expression follows immediately by translating Theorem~\ref{thm:main} into plabic graphs. We sketch a proof here. Details on the combinatorics can be found in \cite{KWZ, Pos}

\begin{proof}[Sketch of proof]
First notice that adding a black lollipop labeled $n$ is precisely the $k$-preserving embedding. This takes care of the first summand. For the second summand, we have four steps.

	(1) The $k$-increasing embedding:  
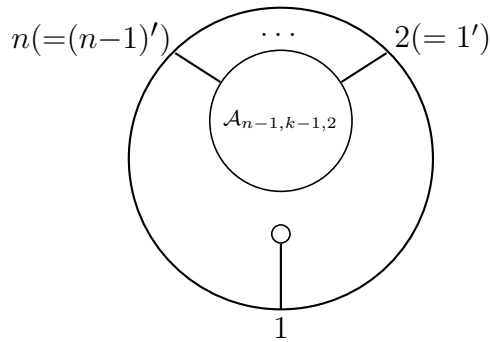
\[
\begin{tikzpicture}[baseline=(current bounding box.center)]
\pgfmathsetmacro{\radius}{2};
\tikzstyle{out1}=[inner sep=0,minimum size=2.4mm,circle,draw=black,fill=black,semithick]
\tikzstyle{in1}=[inner sep=0,minimum size=2.4mm,circle,draw=black,fill=white,semithick]
\draw[thick](0,0)circle[radius=\radius];
\node[circle,draw=black,fill=white,semithick] (Ak) at (0,0.5) {$\scriptstyle \CA_{ n-1,k-1,2}$};

\node[inner sep=0](bn-1)at(135:\radius){};
\node[xshift=-5ex] at(135:\radius+.24){\Large$\scriptstyle n (=(n-1)')$};

\node[inner sep=0](b1)at(135-1*90:\radius){};
\node[xshift=3ex]  at(135-1*90:\radius+.24){$2 (=1')$};

\node[in1](in)at(-90: \radius/2){};
\node[inner sep=0](bn)at(-90:\radius){};
\node at(-90:\radius+.24){$1$};
\path[thick](bn.center)edge(in) (bn-1)edge(Ak) (b1)edge(Ak);
\node[thick] at(90:\radius/2+.65){$\dots$};
\end{tikzpicture}
\]

	(2) Multiplication by $y_1$:
	
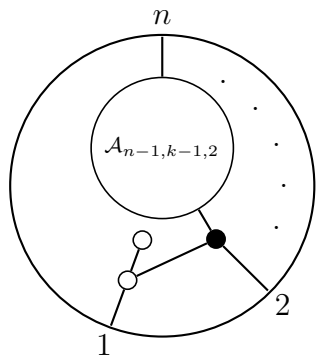
\[
\begin{tikzpicture}[baseline=(current bounding box.center)]
\pgfmathsetmacro{\radius}{2};
\tikzstyle{out1}=[inner sep=0,minimum size=2.4mm,circle,draw=black,fill=black,semithick]
\tikzstyle{in1}=[inner sep=0,minimum size=2.4mm,circle,draw=black,fill=white,semithick]
\draw[thick](0,0)circle[radius=\radius];

\node[circle,draw=black,fill=white,semithick] (Ak) at (0,0.5) {$\scriptstyle \CA_{ n-1,k-1,2}$};

\node[inner sep=0](bn)at(90:\radius){};
\node at(90:\radius+.24){\Large$\scriptstyle n $};

\node[inner sep=0](b2)at(-45:\radius){};
\node  at(-45:\radius+.24){$2$};

\node[out1](s2)at(-45:\radius/2){};

\node[in1](i1)at(250: \radius/3+0.1){};

\node[inner sep=0](b1)at(250:\radius){};
\node at(250:\radius+.24){$1$};

\node[in1](s1)at(250: \radius*2/3){};

\path[thick](bn)edge(Ak) (b2)edge(s2) (s2)edge(Ak) (i1)edge(s1) (s1)edge(b1) (s1)edge(s2);
\node[thick] at(60:\radius/2+.6){$.$};
\node[thick] at(40:\radius/2+.6){$.$};
\node[thick] at(20:\radius/2+.6){$.$};
\node[thick] at(0:\radius/2+.6){$.$};
\node[thick] at(-20:\radius/2+.6){$.$};
\end{tikzpicture}
\]

	(3) Multiplication by $y_2$: 
	
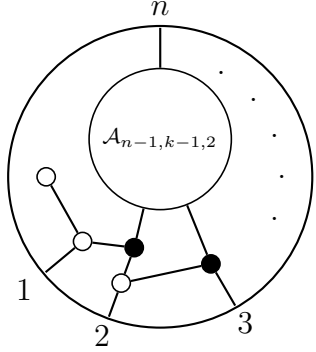
\[
	\begin{tikzpicture}[baseline=(current bounding box.center)]
\pgfmathsetmacro{\radius}{2};
\tikzstyle{out1}=[inner sep=0,minimum size=2.4mm,circle,draw=black,fill=black,semithick]
\tikzstyle{in1}=[inner sep=0,minimum size=2.4mm,circle,draw=black,fill=white,semithick]
\draw[thick](0,0)circle[radius=\radius];

\node[circle,draw=black,fill=white,semithick] (Ak) at (0,0.5) {$\scriptstyle \CA_{ n-1,k-1,2}$};

\node[inner sep=0](bn)at(90:\radius){};
\node at(90:\radius+.24){\Large$\scriptstyle n $};

\node[inner sep=0](b3)at(-60:\radius){};
\node  at(-60:\radius+.24){$3$};
\node[out1](t3)at(-60:\radius*2/3){};

\node[inner sep=0](b2)at(-110:\radius){};
\node  at(-110:\radius+.24){$2$};

\node[out1](s2)at(-110:\radius/2){};

\node[in1](t2)at(-110:\radius*3/4){};

\node[in1](i1)at(180: \radius/2+0.5){};

\node[inner sep=0](b1)at(220:\radius){};
\node at(220:\radius+.34){\Large $\scriptstyle 1$};

\node[in1](s1)at(220: \radius*2/3){};

\path[thick](bn)edge(Ak) (b2)edge(t2) (t2)edge(s2) (s2)edge(Ak) (i1)edge(s1) (s1)edge(b1) (s1)edge(s2) (t2)edge(t3) (Ak)edge(t3) (t3)edge(b3);
\node[thick] at(60:\radius/2+.6){$.$};
\node[thick] at(40:\radius/2+.6){$.$};
\node[thick] at(20:\radius/2+.6){$.$};
\node[thick] at(0:\radius/2+.6){$.$};
\node[thick] at(-20:\radius/2+.6){$.$};
\end{tikzpicture}
\]

	(4) Applying the rotation $\sigma^{-2}_{k,n}$:
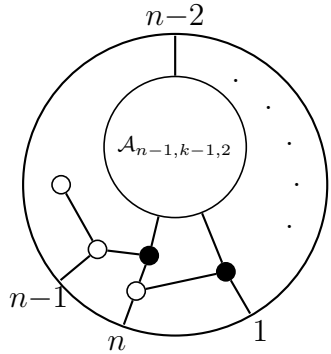
	\[
	\begin{tikzpicture}[baseline=(current bounding box.center)]
\pgfmathsetmacro{\radius}{2};
\tikzstyle{out1}=[inner sep=0,minimum size=2.4mm,circle,draw=black,fill=black,semithick]
\tikzstyle{in1}=[inner sep=0,minimum size=2.4mm,circle,draw=black,fill=white,semithick]
\draw[thick](0,0)circle[radius=\radius];

\node[circle,draw=black,fill=white,semithick] (Ak) at (0,0.5) {$\scriptstyle \CA_{ n-1,k-1,2}$};

\node[inner sep=0](bn)at(90:\radius){};
\node at(90:\radius+.24){\Large$\scriptstyle n-2 $};

\node[inner sep=0](b3)at(-60:\radius){};
\node  at(-60:\radius+.24){$1$};
\node[out1](t3)at(-60:\radius*2/3){};

\node[inner sep=0](b2)at(-110:\radius){};
\node  at(-110:\radius+.24){$n$};

\node[out1](s2)at(-110:\radius/2){};

\node[in1](t2)at(-110:\radius*3/4){};

\node[in1](i1)at(180: \radius/2+0.5){};

\node[inner sep=0](b1)at(220:\radius){};
\node at(220:\radius+.34){\Large $\scriptstyle n-1$};

\node[in1](s1)at(220: \radius*2/3){};

\path[thick](bn)edge(Ak) (b2)edge(t2) (t2)edge(s2) (s2)edge(Ak) (i1)edge(s1) (s1)edge(b1) (s1)edge(s2) (t2)edge(t3) (Ak)edge(t3) (t3)edge(b3);
\node[thick] at(60:\radius/2+.6){$.$};
\node[thick] at(40:\radius/2+.6){$.$};
\node[thick] at(20:\radius/2+.6){$.$};
\node[thick] at(0:\radius/2+.6){$.$};
\node[thick] at(-20:\radius/2+.6){$.$};
\end{tikzpicture}
\]
Finally, it is straightforward to check that this plabic graph corresponds to the same collection of cells as the desired graph. 
\end{proof}

\bibliography{Ref}
\bibliographystyle{amsalpha}

\end{document}